\documentclass[letterpaper, 11pt,  reqno]{amsart}

\usepackage{graphicx}
\usepackage{subfigure}

\usepackage{amssymb,amsmath,amsthm,amscd}

\usepackage{amsxtra, esint, bbm, enumerate}
\usepackage{mathtools} 
\usepackage{skak, bbm} 
\usepackage{ulem}

\usepackage{outlines}
\usepackage{dsfont}

\usepackage{hyperref}

\usepackage{mparhack}

\allowdisplaybreaks[2]

\newtheorem{theorem}{Theorem} [section]

\newtheorem{lemma}[theorem]{Lemma}
\newtheorem{proposition}[theorem]{Proposition}
\newtheorem{remark}[theorem]{Remark}

\newtheorem{definition}[theorem]{Definition}
\newtheorem{corollary}[theorem]{Corollary}

\newtheorem*{ackno}{Acknowledgment}

\newcommand{\I}{\mathcal{I}}

\newcommand{\noi}{\noindent}
\newcommand{\Z}{\mathbb{Z}}
\newcommand{\R}{\mathbb{R}}

\newcommand{\T}{\mathbb{T}}

\newcommand{\TT}{\mathcal{T}}

\newcommand{\BT}{{\bf T}}

\newcommand{\M}{\mathcal{M}}

\newcommand{\NB}{\mathbb{N}}

\newcommand{\FL}{\mathcal{F}L} 
\newcommand{\FLv}{\overrightarrow{\mathcal{F}L}}

\renewcommand{\H}{\mathcal{H}}

\newcommand{\F}{\mathcal{F}}

\newcommand{\dl}{\delta}

\newcommand{\Dl}{\Delta}
\newcommand{\eps}{\varepsilon}

\newcommand{\G}{\Gamma}
\newcommand{\ld}{\lambda}

\newcommand{\s}{\sigma}
\newcommand{\Si}{\Sigma}
\newcommand{\ft}{\widehat}

\newcommand{\wt}{\widetilde}
\newcommand{\cj}{\overline}

\newcommand{\dt}{\partial_t}

\newcommand{\embeds}{\hookrightarrow}

\newcommand{\ta}{\theta}

\renewcommand{\O}{\Omega}

\newcommand{\les}{\lesssim}
\newcommand{\ges}{\gtrsim}

\newcommand{\jb}[1]
{\langle #1 \rangle}

\newcommand{\ind}{\mathbbm{1}}

\numberwithin{equation}{section}
\numberwithin{theorem}{section}

\newcommand{\abs}[1]{\lvert#1\rvert}

\begin{document}

\title[Norm inflation for vNLW]{Norm inflation for the viscous nonlinear wave equation}



\begin{abstract}
In this article, we study the ill-posedness of the viscous nonlinear wave equation for any polynomial nonlinearity in negative Sobolev spaces. In particular, we prove a norm inflation result above the scaling critical regularity in some cases.
We also show failure of $C^k$-continuity, for $k$ the power of the nonlinearity, up to some regularity threshold.
\end{abstract} 

\author{Pierre de Roubin and Mamoru Okamoto}

\address{
Pierre de Roubin\\
School of Mathematics \\
University of Edinburgh \\
Edinburgh\\ 
 EH14 4AS\\
  United Kingdom \\}

\email{p.m.d.de-roubin@sms.ed.ac.uk}

\address{
Mamoru Okamoto\\
Department of Mathematics \\
Graduate School of Science \\
Osaka University \\
Toyonaka\\
 Osaka 560-0043 \\
 Japan\\}

\email{okamoto@math.sci.osaka-u.ac.jp}

\maketitle

\tableofcontents

\section{Introduction}

We consider the viscous nonlinear wave equation (vNLW):
\begin{equation}
\begin{cases}\label{vNLW}
\dt^{2}u-\Dl u + \sqrt{-\Dl}\dt u = -u^k \\
(u,\dt u)|_{t = 0} = (u_0,u_1)\in \H^s (\M^d), 
\end{cases}
\qquad ( t, x) \in \R \times \M^d, 
\end{equation}

\noi where $d \geq 1$, $k \geq 2$ an integer, $\H^s (\M^d) := H^s (\M^d) \times H^{s-1}(\M^d)$ and $\M = \R$ or $\T$, with $\T = \R / \Z$.
Here, we denote $H^s(\M^d)$ the $L^2$-based Sobolev space.
This problem has been introduced by Kuan-\v{C}ani\'c~\cite{KC} as a prototype to model the interaction between a prestressed, stretched membrane and a viscous, incompressible fluid.

On a mathematical level, the interest in studying \eqref{vNLW} lies on its slight difference with the nonlinear wave equation (NLW)
\begin{equation}
\label{NLW}
\dt^{2}u-\Dl u = -u^k.
\end{equation}

\noi Indeed, \eqref{NLW} is a broadly studied equation, see for instance \cite{BH, CCT, FO20, Kap94, KT98, Leb05, Lin93, Lin96, LS95, Tao99, Xia},
and one might wonder how the term $\sqrt{-\Dl}\dt u $ changes
properties  of the solution. In that context, we first look for symmetries in the equation, and more particularly scaling symmetries. In this case, observe that the change of scales
$$
u(t,x) \to \ld^{\frac{2}{k-1}} u (\ld t , \ld x)
$$
\noi preserves solutions of \eqref{vNLW}. Besides, if we denote $\Dot{H}^s (\R^d)$ the homogeneous Sobolev space, equipped with the norm
$$
\| f \|^2_{\Dot{H}^s } = \int_{\R^d} \abs{\xi}^{2s} \abs{\widehat{f}(\xi)}^2 d\xi,
$$

\noi then $\Dot{H}^s (\R^d)$ is preserved by this change of scales when $s$ is the so-called {\it scaling critical regularity} 
\begin{equation}
\label{EQ:sc}
s_{\rm scal} \coloneqq \frac d2 - \frac{2}{k-1}.
\end{equation}

We say that the Cauchy problem \eqref{vNLW} is well-posed if there exists a unique solution and that this solution is stable with respect to our initial data,
namely,
the solution map is well-defined and continuous.
It is shown in the appendix of \cite{Liu23} that \eqref{vNLW} is well-posed in $H^s (\T^2)$ when $k > 3$ and $s \geq s_{\rm scal}$ or $k=3 $ and $s > s_{\rm scal}$.
On the other hand, the well-posedness is not satisfied, in which case we say the Cauchy problem is ill-posed.
In this context,
Kuan-\v{C}ani\'c \cite{KC} prove that \eqref{vNLW} exhibits norm inflation, which we define later on, in $H^s (\R^d)$ when the nonlinearity $k \geq 3$ is an odd integer and $0 < s < s_{\rm scal}$.
In this article, we first prove the norm inflation for the negative Sobolev space
and the non-smoothness of the solution map.

\begin{theorem}
\label{THM:main}
Let $d \geq 1$ and $k \geq 2$. We have the three following results:
\begin{enumerate}
	\item \label{THM:mainA} Let $s < \min ( s_{\rm scal}, 0)$. Fix $(u_0 , u_1) \in \H^s (\M^d)$. Then, given any $\eps > 0$, there exists a solution $u_\eps$ to \eqref{vNLW} on $\M^d$ and $t_\eps \in (0, \eps)$ such that
\begin{equation}
\label{EQ:NIatGID}
\| (u_\eps (0), \dt u_\eps(0)) - (u_0, u_1) \|_{\H^s} < \eps \quad \text{ and } \quad \| u_\eps (t_\eps) \|_{H^s} > \eps^{-1}.
\end{equation}
	\item \label{THM:mainB} Assume furthermore $2 \leq k \leq 5$. Let $\s \in \R$ and $s < - \frac 1k$. Fix $(u_0 , u_1) \in \H^s (\M^d)$. Then, given any $\eps > 0$, there exists a solution $u_\eps$ to \eqref{vNLW} on $\M^d$ and $t_\eps \in (0, \eps)$ such that
\begin{equation}
\label{EQ:NIwithILOR}
\| (u_\eps (0), \dt u_\eps(0)) - (u_0, u_1) \|_{\H^s} < \eps \quad \text{ and } \quad \| u_\eps (t_\eps) \|_{H^\s} > \eps^{-1}.
\end{equation}
	\item \label{THM:mainC} Assume that $2 \le k \le 5$ and $s < d (\frac 12 - \frac 1k) - \frac 1k$.
	Then, the solution map  of \eqref{vNLW} fails to be $C^k$ in $\H^s(\M^d)$.
\end{enumerate}
\end{theorem}

Let us give some precisions on the vocabulary. Equations \eqref{EQ:NIatGID} and \eqref{EQ:NIwithILOR} represent the so-called {\it norm inflation}. More precisely, since they hold for any initial data, we say that we have norm inflation {\it at general initial data}. Finally, having \eqref{EQ:NIwithILOR} for any $\s < s$ means that we have {\it infinite loss of regularity}.

Besides, norm inflation does imply the ill-posedness, especially the discontinuity of the solution map. More specifically, if we denote
$$
s_M (k,d)  = \max \Big( s_{\rm scal}  , - \frac 1k \Big),
$$

\noi then Theorem~\ref{THM:main} yields the following corollary:


\begin{corollary}
\label{MapDiscont}
Let $d \geq 1$, $k \geq 2$ and $s < \min(s_M (k,d), 0)$. Then, for any $T > 0$, the solution map
of \eqref{vNLW}
is discontinuous everywhere in $\H^s (\M^d)$.

In particular, \eqref{vNLW} is ill-posed in $\H^s (\M^d)$ for $s < \min(s_M, 0)$.
\end{corollary}

Note that, in the cases of quadratic equation in dimensions $1$ and $2$, and cubic equation in dimension $1$, Corollary~\ref{MapDiscont} implies the ill-posedness above the scaling critical regularity. On another note, while the third part of Theorem~\ref{THM:main} does not go so far as discontinuity of the solution map, it does imply that we cannot use an iteration argument to prove well-posedness and hence suggest that our problem might be ill-posed in this range of regularity.
More generally, this lets us assume that
the regularity threshold for the well-posedness in $L^2$-based Sobolev spaces is
$\max (s_{\rm scal}, s_{\rm vis})$,
where 
\begin{equation}
\label{defsWP}
s_{\rm vis} \coloneqq d \bigg( \frac 12 - \frac 1k\bigg) - \frac 1k .
\end{equation}
Note that $s_{\rm vis} > s_{\rm scal}$ holds if and only if $d$ and $k$ satisfy either of
%
%
$d=1$ or ($d=2$ and $k=2$).

This intuition is also strengthened by the following result of the well-posedness in $H^s (\M^d)$ for $s_{\rm vis}< s \leq 0$.

\begin{theorem}
\label{thm:WP}
Let $d$ and $k$ satisfy
\begin{equation}
\big(\text{$d = 1, 2, 3$ and $k = 2$}\big)
\quad \text{or} \quad
\big(\text{$d =1 $ and $k = 3$}\big).
\label{WPA}
\end{equation}
Then, \eqref{vNLW} is locally well-posed in $\H^s (\M^d)$ for $s_{\rm vis} < s \leq 0$.
\end{theorem}

The condition \eqref{WPA} yields that $s_{\rm scal} \le s_{\rm vis}<0$.
Because of the presence of the dissipative effect in \eqref{vNLW},
the well-posedness holds even in the negative Sobolev space.
Note that
the Cauchy problem for the wave equation \eqref{NLW} is ill-posed in $\H^s(\M^d)$ for $s<0$.
See \cite{CCT, FO20}.

\begin{remark}
\rm
We believe that the argument developed by Kapitanski~\cite{Kap94} for the wave equation should be applicable to this equation, and then give well-posedness for \eqref{vNLW} in $H^s (\M^d)$ for $s > \max (s_{\rm scal} , 0)$.
See also Remark \ref{REM:wp1d} below and the appendix of \cite{Liu23} for the one and two dimensional cases, respectively.
\end{remark}


Let us give some details on norm inflation.
Christ-Colliander-Tao \cite{CCT} introduced norm inflation for the first time with a dispersionless ODE approach. See also \cite{BTz1, Xia, OW}.
Carles and his collaborators \cite{AC, BC, CK} also proved norm inflation with loss of regularity for the nonlinear Schr\"odinger equation by using geometric optics.

Bejenaru-Tao \cite{BT} used a Fourier analytic approach to prove the ill-posedness.
The idea is to expand the solution into a power series and use a {\it high-to-low energy transfer} in one of the terms to prove ill-posedness. Iwabuchi-Ogawa~\cite{IO} refined the idea afterwards to extend the ill-posedness result into a norm inflation result. This method was taken over by Kishimoto~\cite{Kis19} who developed it to prove norm inflation for the nonlinear Schr\"odinger equation, while Oh~\cite{Oh} introduced a way to index the power series with trees, and estimate all the terms separately. See also \cite{BH2, COW, CP, FO20, OOT, Ok, WZ} and \cite{Chevyrev} for an implementation of this argument in a probabilistic setting.

The worst interaction for \eqref{vNLW} in positive Sobolev spaces is the low-to-high energy transfer.
By using the ODE approach,
Kuan-\v{C}ani\'c \cite{KC} proved the norm inflation.
On the other hand,
we use the Fourier analytic approach in this article,
because the high-to-low energy transfer is the worst interaction in negative Sobolev spaces.

Let us discuss our strategy. First, we use the Wiener algebra $\FL^{0,1} (\M^d)$, which we define now. Given $\M = \R$ or $\T$, let $\widehat{\M}^d$ denote the Pontryagin dual of $\M^d$, i.e.
\begin{equation}
\notag
\widehat{\M}^d = 
\begin{cases}
\R^d & \text{if} \quad \M = \R, \\
\Z^d & \text{if} \quad \M = \T.
\end{cases}
\end{equation}

\noi Note that, when $\widehat{\M} = \Z$, we endow it with the counting measure. We can then define the following Fourier-Lebesgue spaces:

\begin{definition}[Fourier-Lebesgue spaces] \rm
\label{DEF:FLspaces}
For $s \in \R$ and $p \geq 1$, we define the Fourier-Lebesgue space $\FL^{s,p} (\M^d)$ as the completion of the Schwartz class of functions $\mathcal{S} ( \M^d)$ with respect to the norm
$$
\| f \|_{\FL^{s,p}(\M^d)} = \| \jb{\xi}^s \widehat{f}(\xi) \|_{L^p_\xi (\widehat{\M}^d)},
$$

\noi where $\jb{\xi} \coloneqq (1 + \abs{\xi}^2 )^{\frac 12}$.
\end{definition}

\noi The algebra property of $\FL^{0,1} (\M^d)$ allows us to prove the well-posedness in $\FL^{0,1} (\M^d)$ of \eqref{vNLW}.
As a result, the solution $u$ to \eqref{vNLW} can be expanded into a power series depending only on the initial data $\vec u_0 := (u_0, u_1)$,
which we denote
$$
u = \sum^\infty_{j = 0} \Xi_j (\vec u_0).
$$

\noi For any $j \geq 0$, $\Xi_j (\vec u_0)$ is a multilinear term in $\vec u_0$ of degree $(k-1)j +1$. In particular, these terms are the successive terms of a Picard iteration expansion. Then, the idea is to show by explicit computations that $\Xi_1 (\vec u_0)$ grows rapidly in a short time, while the other terms are controlled.

From the fact that the dispersive effect as well as the dissipative effect in \eqref{vNLW} is weak for a short time,
we show the lower bound of $\Xi_1 (\vec u_0)$ in a short time interval.
This lower bound holds only for a shorter interval than that in \cite{FO20}, because of the presence of the dissipative effect.
See Proposition \ref{PROP:MultilinearEstPart2}.
Since the same estimates as in \cite{FO20} are applicable to the higher iteration terms in this case,
we then obtain the norm inflation in $\H^s(\M^d)$ for $s<\min (s_{\rm scal},0)$.
Note that the infinite loss of regularity does not follow from this argument.
See Section~\ref{Sec3}.

For the norm inflation above the scaling critical regularity or the infinite loss of regularity,
we need to allow our solutions to have a slightly longer time of existence to show the blow-up of $\Xi_1 (\vec u_0)$.
Since the dispersive and dissipative effects appear,
a more precise calculation (compared to Proposition \ref{PROP:MultilinearEstPart2}) is needed to show another lower bound of $\Xi_1 (\vec u_0)$.
We prove this lower bound by assuming $k \le 5$ to avoid the case where contribution from the initial data is much more complicated.
This condition might be technical,
because we can prove the lower bound for even $k \ge 6$ at least.
However,
even if this condition is removed,
our ill-posedness result is not improved and we thus impose this assumption.
See Remark \ref{REM:evenk}.
On a parallel method inspired by Bourgain~\cite{Bou97},
this lower bound of $\Xi_1 (\vec u_0)$ gives us another result: the failure of $C^k$-continuity for the solution map.

We reach a point where the growth of the tail of the power series becomes important to show a clear dominance of $\Xi_1 (\vec u_0)$ and the new bound $s < - \frac 1k$ appears.
More precisely,
the condition to be well-posed in $\FL^{0,1} (\M^d)$ is too strong to show the norm inflation for some $s>s_{\rm scal}$.
We then directly estimate each of higher order iteration terms as in \cite{MO15, MO16}.
This argument works well for the quadratic case in $d=1,2,3$.
However, we can not obtain the norm inflation in $\H^s (\M^d)$ for $\max(s_{\rm scal}, -\frac 1k) \le s<s_{\rm vis} \, (\le 0)$ when $d=1$ and $k=3,4$,
since a condition on the convergence of the power series is too strong.
See Remark \ref{REM:ni34}.
Even in the positive Sobolev spaces, this gap appears.
Namely,
it is unclear whether \eqref{vNLW} is well-posed or ill-posed in $\H^s(\M)$ for $\max(s_{\rm scal},-\frac 1k) \le s \le s_{\rm vis}$
in the one dimensional case,
while
the third part of Theorem  \ref{THM:main}
implies that
the well-posedness in $\H^s(\M)$ does not follow from an iteration argument
at least
for $k \le 5$ and $s< s_{\rm vis}$.

Finally, we conclude this article with the proof of Theorem~\ref{thm:WP}.
Since the Schauder-type estimate is proved in Liu-Oh~\cite{LO22},
the standard argument yields the well-posedness in negative Sobolev spaces.
However, we give a proof of the well-posedness in Section~\ref{sec:WPofvNLW}, for completeness.

\section{Preliminary results and tree structure}

In this section,
following the argument as in \cite{Oh} (see also \cite{FO20}),
we prove the well-posedness for \eqref{vNLW} in the Wiener algebra $\FL^{0,1}(\M^d)$ and that the solution to our problem can be expanded into a power series. 

\subsection{Duhamel formulation and some notations}

First, we give in this subsection some notations for clarity. Taking the Fourier transform of the equation, we get the following Duhamel formulation:
\begin{equation}
\notag
\begin{split}
 \widehat{u}(t, \xi) = e^{- \frac 12 \abs{\xi}t} \bigg[ \cos\bigg(\frac{\sqrt{3}}{2}\abs{\xi} t\bigg) & + \frac{1}{\sqrt{3}} \sin\bigg(\frac{\sqrt{3}}{2} \abs{\xi}t\bigg) \bigg] \widehat{u}_0 (\xi) +  e^{- \frac 12 \abs{\xi}t} \frac{\sin\big(\frac{\sqrt{3}}{2}\abs{\xi} t\big)}{\frac{\sqrt{3}}{2}\abs{\xi}}  \widehat{u}_1 (\xi) 
	\\ & - \int^t_0  e^{- \frac 12 \abs{\xi}(t-t')} \frac{\sin\big(\frac{\sqrt{3}}{2}\abs{\xi}(t-t')\big)}{\frac{\sqrt{3}}{2}\abs{\xi}} \widehat{u^k}(t', \xi) \mathrm{d}t'.
\end{split}
\end{equation}

\noi Let us denote
$$
D = \F^{-1}_x ( \abs{\xi}).
$$

\noi Then, let $P(t)$ denote the Poisson kernel
\begin{equation}
\label{EQ:DefiPoissonKer}
P(t) = e^{-\frac 12 Dt}
\end{equation}

\noi and let $V_1(t)$ be defined by
\begin{equation}
\notag
V_1(t) = \frac{\sin\big(\frac{\sqrt{3}}{2} D t\big)}{\frac{\sqrt{3}}{2} D}.
\end{equation}

\noi We also denote 
\begin{equation}
\label{EQ:DefiW}
W(t) = P(t) \circ V_1(t) = e^{-\frac 12 Dt}\frac{\sin\big(\frac{\sqrt{3}}{2} D t\big)}{\frac{\sqrt{3}}{2} D}
\end{equation}

\noi and, for any appropriate functions $u_1, \dots, u_k$,
\begin{equation}
\label{EQ:DefiDuhamelOp}
I_k (u_1, \dots, u_k)(t)
= -\int^t_0
W(t - t')
\prod^k_{j = 1} u_j (t') \mathrm{d}t'.
\end{equation}

\noi Similarly, let us denote
\begin{equation}
\label{EQ:DefiU}
U(t)(u_0, u_1) = \bigg[ \cos\bigg(\frac{\sqrt{3}}{2}D t\bigg) + \frac{1}{\sqrt{3}} \sin\bigg(\frac{\sqrt{3}}{2} D t\bigg) \bigg] u_0 + V_1(t) u_1 \eqqcolon V_0 (t) u_0 + V_1 (t) u_1
\end{equation}

\noi and 
\begin{equation}
\label{EQ:DefiV}
V(t) = P(t) \circ U(t).
\end{equation}

\noi Therefore, we can write the Duhamel formulation of \eqref{vNLW} as
\begin{equation}
\label{EQ:DuhamelNotations}
u(t) = V(t) \vec{u}_0 + I_k (u)(t)
\end{equation}

\noi where $\vec u_0 \coloneqq (u_0, u_1)$ and we used the shorthand notation $I_k(u) = I_k (u, \dots, u)$.

\subsection{Tree structure and power series expansion}

In order to prove norm inflation, we want to expand the solution to \eqref{vNLW} into a power series. To do so, we exploit the multilinearity of the Duhamel operator $I_k$ defined in \eqref{EQ:DefiDuhamelOp} and iterate the Duhamel formula \eqref{EQ:DuhamelNotations}. From this process, the power series we get is naturally indexed by a tree structure.
The goal of this subsection is to define this tree structure and state a power series expansion depending only on the initial data. First, we have the following definition of trees:
 \begin{definition} [$k$-ary trees]
\label{DEF:Trees}
\rm 
\begin{enumerate}
	\item Given a set $\TT$ with partial order $\leq$, we say that $b \in \TT$, with $b \leq a$ and $b \ne a$ for some $a \in \TT$, is a child of $a$ if  $b\leq c \leq a$ implies either $c = a$ or $c = b$. If the latter condition holds, we also say that $a$ is the parent of $b$.

	\item A tree $\TT$ is a finite partially ordered set,  satisfying the following properties\footnote{We do not identify two trees even if there is an order-preserving bijection between them.}:
\begin{itemize}
	\item Let $a_1, a_2, a_3, a_4 \in \TT$. If $a_4 \leq a_2 \leq a_1$ and  $a_4 \leq a_3 \leq a_1$, then we have $a_2\leq a_3$ or $a_3 \leq a_2$.

	\item A node $a\in \TT$ is called terminal if it has no child. A non-terminal node $a\in \TT$ is a node with exactly $k$ children.

	\item There exists a maximal element $r \in \TT$ (called the root node) such that $a \leq r$ for all $a \in \TT$.

	\item $\TT$ consists of the disjoint union of $\TT^0$ and $\TT^\infty$, where $\TT^0$ and $\TT^\infty$denote the collections of non-terminal nodes and terminal nodes, respectively.
\end{itemize}

	\item Let $\BT(j)$ denote the set of all trees with $j$ non-terminal nodes.
\end{enumerate}
\end{definition} 

\noi Note that a tree of $j$-th generation $\TT \in \BT(j)$ has $kj + 1$ nodes, of which $j$ are non-terminal nodes. Besides, we have the following bound on the number of trees of $j$-th generation (See Lemma 2.3 in \cite{Oh}):

\begin{lemma}
\label{LEM:NumberOfTrees}
There exists a constant $C_0 > 0$ such that, for any $j \in \NB_0 := \NB \cup \{0\}$, we have
\begin{equation}
\notag
\abs{\BT (j)} \leq C^j_0.
\end{equation}
\end{lemma}

From now on, for any $\vec{\phi} \in \FLv^1 (\M^d) \coloneqq \FL^{0,1}(\M^d) \times \FL^{-1,1} (\M^d)$, we associate to any tree $\TT \in \BT (j)$, $j \geq 0$, a space-time distribution $\Psi (\TT) (\vec{\phi}) \in \mathcal{D}' \left( (-T,T) \times \M^d \right)$ as follows:
\begin{itemize}
	\item replace a non-terminal node by Duhamel integral operator $\I_k$ with its $k$ arguments being the children of the node,
	\item replace a terminal node by $V(t) \vec{\phi}$.
\end{itemize}

\noi For any $j \geq 0$ and $\vec{\phi} \in \FLv^1(\M^d)$, we also define $\Xi_j$ as follows : 
\begin{equation}
\notag
\Xi_j (\vec{\phi}) = \sum_{\TT \in \BT (j)} \Psi(\TT)( \vec{\phi} ).
\end{equation}

\noi Let us now state the local well-posedness in the Wiener algebra $\FL^1(\M^d)$.
Since
$e^{-\frac 12|\xi|t} \le 1$,
the proof is the same as Lemma 2.4 in \cite{FO20}
and we thus omit the proof here.

\begin{lemma}
\label{LEM:ExistenceOfSolution}
Let $M > 0$. Then, for any time $T$ such that $0 < T \ll \min(M^{-\frac{k-1}{2} }, 1)$ and $\vec{u}_0 \in \FLv^1(\M^d)$ with $\| \vec{u}_0 \|_{\FLv^1} \leq M$, 
we have the following:
\begin{enumerate}
	\item there exists a unique solution $u \in C ([0,T]; \FL^1 (\M^d))$ to our problem \eqref{vNLW} satisfying $(u, \dt u)|_{t = 0} = \vec{u}_0$.
	\item Moreover, u can be expressed as
\begin{equation}
\notag
u = \sum_{j = 0}^\infty \Xi_j (\vec{u}_0) = \sum_{j = 0}^\infty \sum_{\TT \in \BT (j)} \Psi(\TT)(\vec{u}_0)
\end{equation}
in $C ([0,T]; \FL^1 (\M^d))$.
\end{enumerate}

\end{lemma}

\section{A first result of norm inflation}
\label{Sec3}

In this section, we give the proof for a first result of norm inflation, namely we prove part \ref{THM:mainA} of Theorem \ref{THM:main} under the additional conditions
\begin{equation}
\label{eq:condspart1}
\bigg\{- \frac d2 < s < \min(s_{\rm scal}, 0) \text{ and } k > 1 + \frac 2d \bigg\} \quad \text{ or } \quad \bigg\{s \leq - \frac d2 \text{ and } k \geq 5 \bigg\}.
\end{equation}
\noi Note that the case $s \le - \frac d2$ and $k \leq 4$ is implied by the second part of Theorem~\ref{THM:main}, which we prove in Section~\ref{Sec4}.
See Remark~\ref{REM:ProofCompletes}. We also recall that $s_{\rm scal}$ is the scaling critical regularity given in \eqref{EQ:sc}.

In order to prove this section's result, our main goal is really to prove the following proposition.

\begin{proposition}
\label{PROP:NIatGIDn}
Let $ d \geq 1$ and $(s,k)$ satisfying the assumption \eqref{eq:condspart1}. Fix $(u_0, u_1) \in \mathcal{S} ( \M^d) \times \mathcal{S}(\M^d)$. Then, for any $n \in \NB$, there exists a solution $u_n$ to \eqref{vNLW} and $t_n \in (0, \frac 1n )$ such that
\begin{equation}
\notag
\| (u_n (0), \dt u_n (0)) - (u_0 , u_1) \|_{\H^s} < \frac 1n \quad \text{ and } \quad \| u_n (t_n ) \|_{H^s} > n.
\end{equation}
\end{proposition}

Indeed, assume Proposition \ref{PROP:NIatGIDn} is proved and fix $(u_0, u_1) \in \H^s (\M^d)$ and $\eps > 0$. By density, there exist $(v_0, v_1) \in \mathcal{S} ( \M^d) \times \mathcal{S}(\M^d)$ such that
$$
\| (v_0, v_1) - (u_0 , u_1) \|_{\H^s} < \frac \eps 2.
$$

\noi Besides, choose $n \in \NB$ such that $n > \frac{2}{\eps}$. According to Proposition~\ref{PROP:NIatGIDn}, there exists a solution $u_n$ to \eqref{vNLW} and a time $t_n \in (0, \frac 1n ) \subset (0, \eps)$ such that 
$$
\| (u_n (0), \dt u_n (0)) - (v_0 , v_1) \|_{\H^s} < \frac 1n < \frac{\eps}{2} \quad \text{ and } \quad \| u_n (t_n ) \|_{H^s} > n > \eps^{-1},
$$

\noi and our result follows by a triangular inequality.

\subsection{Preliminary notations and multilinear estimates}

In this subsection, we construct the sequence solutions $(u_n )_{n \in \NB}$ mentioned in Proposition \ref{PROP:NIatGIDn} by choosing specific initial data. To do so, we work with parameters $N = N(n) \gg 1$, $R = R(N)$ and $A = A(N) \ll N$ and we use them to define $\vec{\phi}_n = (\phi_n, 0)$ by

\begin{equation}
\label{EQ:Defiphi_n}
\widehat{\phi_n} = R \ind_\O
\end{equation}

\noi where 

\begin{equation}
\label{EQ:DefiOmega}
\O = \bigcup_{\eta \in \Si} (\eta + Q_A)
\end{equation}

\noi with $Q_A = [-\frac A8 , \frac A8]^d$,

\begin{equation}
\label{EQ:DefiSigma}
\Si = \{ -2N e_1, -N e_1, N e_1, 2N e_1\}
\end{equation}

\noi and $e_1 = (1,0, \dots,0) \in \ft{\M}^d$. Here, $\ind_\O$ denotes the characteristic function of the set $\O$. Note that the condition $A \ll N$ ensures the union \eqref{EQ:DefiOmega} is disjoint. We also denote

\begin{enumerate}
	\item $\vec{u}_{0,n} = \vec{u}_0 + \vec{\phi}_n$ where $\vec{u}_0 =(u_0,u_1)$ is fixed,
	\item $u_n$ the solution to \eqref{vNLW} with initial data $\vec{u}_{0,n}$.
\end{enumerate}

\noi Since $\vec{u}_0$ is fixed and smooth, let us assume that 
\begin{equation}
\notag
\|\vec u_0 \|_{\H^0} \sim \| \vec{u}_0 \|_{\FLv^1} \sim 1 \ll RA^d = \| \vec{\phi}_n \|_{\FLv^1}.
\end{equation}

\noi Now, suppose that 

\begin{equation}
\label{EQ:CondOnT1}
0 < T \ll \min \big( (RA^d)^{-\frac{k-1}{2}} , 1 \big).
\end{equation}

\noi Then, according to Lemma \ref{LEM:ExistenceOfSolution}, there exists a unique solution to \eqref{vNLW} with initial data $\vec{u}_{0, n}$, and we can expand this solution into the following power series:
\begin{equation}
\label{EQ:DecOfun}
u_n = \Xi_0 (\vec{u}_{0,n})+  \Xi_1 (\vec{\phi}_{n}) + \left( \Xi_1 (\vec{u}_{0,n})- \Xi_1 (\vec{\phi}_{n}) \right)  +  \sum^{\infty}_{j = 2} \Xi_j (\vec{u}_{0,n}).
\end{equation}

\noi The idea is to estimate each of these terms. First, we get the following lemma.
Since the proof is a repetition of Lemma 3.2 in \cite{FO20},
we omit the proof here.

\begin{lemma}
\label{LEM:MEstPart1}

For any $ s <0$, $t \ge 0$ and $j \in \NB$, the following estimates hold:
\begin{align}
\| \vec{u}_{0,n} - \vec{u}_0 \|_{\H^s} & \sim RN^s A^{\frac d2}, \label{EQ:EstID} \\
\| \Xi_0 (\vec{u}_{0,n})(t) \|_{H^s} & \les 1 + R N^sA^{\frac d2},
\notag
\\
\| \Xi_1 (\vec{u}_{0,n})(t) - \Xi_1 (\vec{\phi}_n) (t) \|_{H^s} & \les t^2 (R A^d)^{k-1}, 
\notag
\\
\| \Xi_j (\vec{u}_{0,n})(t) \|_{H^s} & \les C^j t^{2j} (R A^d)^{(k-1)j} \left( 1 +  R g_s(A) \right) \label{EQ:MEst3}
\end{align}

\noi where $g_s (A)$ is defined by 
\begin{equation}
\label{DEF:gs}
g_s (A) \coloneqq
\begin{cases}
 1 & \textup{if } s<-\frac{d}{2}, \\
 \left( \log \jb{A} \right)^{\frac{1}{2}}& \textup{if } s=-\frac{d}{2}, \\
 A^{\frac{d}{2}+s} & \textup{if }  s>-\frac{d}{2}.
\end{cases}
\end{equation}

\end{lemma}

We now recall some bounds on convolutions of characteristic functions of cubes:

\begin{lemma}
\label{LEM:ConvolutionIneq}
Let $a,b \in \R^d$ and $A > 0$, then we have
$$
C_d A^d \ind_{a + b + Q_A } (\xi) \leq  \ind_{a + Q_A } \ast \ind_{b + Q_A } (\xi) \leq \widetilde{C}_d A^d \ind_{a + b + Q_{2A}}(\xi )
$$

\noi where $C_d, \widetilde{C}_d > 0$ are constants depending only on the dimension $d$.
\end{lemma}

Using this lemma, we can prove the following proposition:

\begin{proposition}
\label{PROP:MultilinearEstPart2}
Let $ s < 0$, $\vec{\phi}_n$ defined as in \eqref{EQ:Defiphi_n} with $1 \le A \ll N $ and $0<t \ll N^{-1}$. Then, we have
\begin{equation}
\label{EQ:MEst4}
\| \Xi_1 (\vec{\phi}_n) (t) \|_{ H^s} \ges R^k t^2 A^{d(k - 1) }g_s(A)
\end{equation}

\noi with $g_s(A)$ defined by \eqref{DEF:gs}.
\end{proposition}

\begin{proof}
In order to make this proof easier to read, we denote 
\begin{equation}
\notag
\G_\xi \coloneqq \left\{ (\xi_1, \dots, \xi_k) \in \widehat{\M}^k \, \bigg| \ \sum^k_{j=1} \xi_j = \xi \right\} \quad \text{ and } \quad \mathrm{d}\xi_\G \coloneqq \mathrm{d}\xi_1 \cdots \mathrm{d}\xi_{k-1}.
\end{equation}

\noi Then, we can write using this, the fact that $\vec{\phi}_n = (\phi_n , 0)$ and equations \eqref{EQ:DefiDuhamelOp}, \eqref{EQ:DefiV}, \eqref{EQ:Defiphi_n}, \eqref{EQ:DefiOmega} and \eqref{EQ:DefiSigma}:
\begin{align*}
\F_x \big[ \Xi_1 (\vec{\phi}_n) \big] (t, \xi) & = I_k \big( V(t) \vec{\phi}_n \big) (t, \xi) \\
	& = -R^k \sum_{(\eta_1 , \dots, \eta_k) \in \Si^k} \int^t_0 e^{-\frac 12 \abs{\xi} (t-t')} V_1(t - t', \xi)
\\
&\qquad
\times
\int_{\G_\xi} \prod^k_{j = 1}
	\Big( e^{-\frac 12 \abs{\xi_j} t'} V_0 (t, \xi_j) \ind_{\eta_j + Q_A} (\xi_j) \Big) \mathrm{d}\xi_\G \mathrm{d}t'.
\end{align*}
We have
\[
\| \Xi_1 (\vec{\phi}_n) (t) \|_{H^s}
\ges
\| \jb{\cdot}^s \F_x [ \Xi_1 (\vec{\phi}_n)] (t) \ind_{Q_A} \|_{L^2_\xi}.
\]
When
$\xi \in Q_A$,
the summation in $\F_x \big[ \Xi_1 (\vec{\phi}_n) \big] (t, \xi)$
is restricted to $\eta_1+ \cdots +\eta_k = 0$.
We make a few observations:
\begin{outline}
	\1 Since we assume $A \ll N$, for any $j \in \{ 1, \dots, k\}$, $\abs{\xi_j} \sim N$.
	\1 Since we assume $0<t \ll N^{-1}$, and we have $t' \in [0,t]$ and $\abs{\xi_j} \sim N$, for any $j \in \{ 1, \dots, k\}$, $t' \abs{\xi_j} \ll 1$. This implies 
		\2 $e^{-\frac 12 \abs{\xi_j} t'} \geq \frac 12$.
		\2 $\cos\bigg(\frac{\sqrt{3}}{2}\abs{\xi_j} t'\bigg) \geq \frac 12$.
		\2 $\frac{1}{\sqrt{3}} \sin\bigg(\frac{\sqrt{3}}{2} \abs{\xi_j}t'\bigg) \sim \abs{\xi_j}t'$.
		\2 Combining the last two estimates, we get $V_0 (t', \xi_j) \ges 1$, where $V_0$ is defined in \eqref{EQ:DefiU}.
	\1 Since $(\eta_1, \dots, \eta_k) \in \Si_1$, $\eta_1+ \cdots +\eta_k = 0$, then $ \xi = \xi_1 + \cdots + \xi_k \in Q_A$. Namely, we get
$$
\abs{\xi} \les A \ll N.
$$
	\1 Since we assume $0<t \ll N^{-1}$, and we have $t' \in [0,t]$ and $\abs{\xi} \ll A$, then $(t- t')\abs{\xi} \ll 1$. This implies
		\2 $e^{-\frac 12 \abs{\xi} (t-t')} \geq \frac 12$,
		\2 $V_1(t - t', \xi) = \frac{\sin\big(\frac{\sqrt{3}}{2}\abs{\xi} (t-t')\big)}{\frac{\sqrt{3}}{2}\abs{\xi}} \sim t - t'$.
\end{outline}

\noi Therefore we get from these observations and the lefthand side inequality of Lemma \ref{LEM:ConvolutionIneq}
\begin{align*}
&| \F_x [ \Xi_1 (\vec{\phi}_n)] (t,\xi) \ind_{Q_A} (\xi)|
\\& \ges
R^k
\sum_{\substack{(\eta_1 , \dots, \eta_k) \in \Si^k \\ \eta_1+ \dots + \eta_k =0}} \int^t_0 \frac 12 (t-t') \int_{\G_\xi} \prod^k_{j=1}
\Big( \frac 12 \times \frac 12 \times \ind_{\eta_j + Q_A} (\xi_j) \Big) \mathrm{d}\xi_\G \mathrm{d}t' 
\ind_{Q_A} (\xi) \\
& \ges
R^k t^2 A^{d(k-1)} \ind_{Q_A} (\xi).
\end{align*}

\noi
Taking the $H^s$ norm, this implies
\eqref{EQ:MEst4}.
\end{proof}

\subsection{Proof of Proposition \ref{PROP:NIatGIDn}}
\label{Subsec3.2}

The goal of this subsection is to prove Proposition~\ref{PROP:NIatGIDn} by using the multilinear estimates proved before. Namely, we want to choose $R$, $A$ and $T$ such that Lemma \ref{LEM:ExistenceOfSolution} and Lemma \ref{LEM:MEstPart1} apply, and we have 
\begin{equation}
\notag
\| (u_n (0), \dt u_n (0)) - (u_0 , u_1) \|_{\H^s} < \frac 1n \quad \text{ and } \quad \| u_n (t_n ) \|_{H^s} > n.
\end{equation}

\noi Note first that \eqref{EQ:EstID} implies the following condition:

\begin{equation}
\label{EQ:cond1}
R N^s A^{\frac d2} < \frac 1n.
\end{equation}

\noi Besides, \eqref{EQ:DecOfun}, Lemma \ref{LEM:MEstPart1} and Proposition \ref{PROP:MultilinearEstPart2}, give
$$
\| u_n (T) \|_{H^s} \ges T^2 R^k A^{d(k- 1)} g_s (A) -
\Big( 1 + RN^s A^{\frac d2} \Big) - T^2 \big(RA^d\big)^{k-1} - \sum^{\infty}_{j = 2} \|\Xi_j (\vec{u}_{0,n})(T) \|_{H^s}.
$$

\noi Let us focus a bit on $\sum^{\infty}_{j = 2} \|\Xi_j (\vec{u}_{0,n})(T) \|_{H^s}$. Let us assume that
\begin{equation}
\label{EQ:cond2}
R g_s(A) \gg 1.
\end{equation}

\noi This, combined with \eqref{EQ:MEst3} and the assumption \eqref{EQ:CondOnT1}, gives
$$
\sum^{\infty}_{j = 2} \|\Xi_j (\vec{u}_{0,n})(T) \|_{H^s} \les C^2 T^4 (RA^d)^{2(k-1)} R g_s(A).
$$

\noi Therefore, we need to choose $T$, $R$ and $A$ so that $T^2 R^k A^{d(k-1)} g_s (A) > n$ and the other terms are smaller than this. Namely, we want the following:

\begin{itemize}
	\item the new assumption
\begin{equation}
\label{EQ:Cond3}
T^2 R^k A^{d(k- 1)}g_s(A) \gg n,
\end{equation}
	\item $T^2 R^k A^{d(k- 1)}g_s(A) \gg RN^s A^{\frac d2}$ which is implied by \eqref{EQ:cond1} and \eqref{EQ:Cond3},
	\item $T^2 R^k A^{d(k-1)}g_s(A) \gg T^2 \big(RA^d\big)^{k-1}$ which is equivalent to $Rg_s(A) \gg 1$ and this is already assumed in \eqref{EQ:cond2},
	\item $T^2 R^k A^{d(k- 1)}g_s(A) \gg T^4 (RA^d)^{2(k-1)} R g_s(A)$ which means $1 \gg T^2 (RA^d)^{k-1}$ and this is already assumed in \eqref{EQ:CondOnT1}.
\end{itemize}

\noi Therefore, it is sufficient to find $T$, $R$ and $A$ such that

\begin{align*}
\textup{(i)} & \quad T \ll N^{-1} \text{ and } T \ll \big(RA^d)^{-\frac{k-1}{2}}, \\
\textup{(ii)} & \quad 1 \le A \ll N, \\
\textup{(iii)} & \quad 1 \ll R g_s(A) \text{ and } 1 \ll RA^d, \\
\textup{(iv)} & \quad R N^sA^{\frac d2} < \frac 1n, \\
\textup{(v)} & \quad T^2 R^k A^{d(k-1)}g_s(A)\gg n.
\end{align*}

\noi We split the remaining of the proof in three cases.

{\bf Case 1:} Assume $s < - \frac d2$ and $k \geq 5$. Observe then that we have $s + \frac 12 < 0$. Let $0 < \dl < \min (2, - \frac 43 (\frac 12 + s))$ and set
$$
R = N^\dl \quad \text{ and } \quad A = N^{\frac 1d (1 - \frac 12 \dl)}.
$$

\noi Then, conditions (ii), (iii) and (iv) are satisfied, while the other conditions can be reduced to 
\begin{align*}
\textup{(i')} & \quad T \ll N^{-1} \text{ and } T \ll N^{- \frac{k-1}2 - \frac{k-1}4 \dl}, \\
\textup{(v')} & \quad T^2 N^{k-1 + \frac{k+1}2 \dl} \gg n.
\end{align*}

\noi Then, it suffices to choose $T$ satisfying
\[
N^{- \frac{k-1}2 - \frac{k+1}4 \dl}
\ll T \ll \min \big( N^{-1}, N^{- \frac{k-1}2 - \frac{k-1}4 \dl} \big)
\]

\noi which is possible since $\dl>0$ and $k \geq 5$.

{\bf Case 2:} Assume $s = - \frac d2$ and $k \geq 5$. Let $0 < \dl < \frac 1{16(k-1)}$ and set 
$$
R = (\log N)^\dl \quad \text{ and } \quad A = N^{\frac 1d } (\log N)^{- \frac 1{8(k-1)d}}.
$$

\noi Then, conditions (ii), (iii) and (iv) are again satisfied and, in a similar way as before, the other conditions reduce to 
\begin{align*}
\textup{(i')} & \quad T \ll N^{-1} \text{ and } T \ll N^{- \frac{k-1}2}(\log N)^{- \frac{k-1}2 \dl + \frac 1{16}}, \\
\textup{(v')} & \quad  T^2 N^{k-1}(\log N)^{k \dl + \frac 3{8}} \gg n.
\end{align*}

\noi Indeed, observe that
\begin{align*}
T^2 R^k A^{d(k-1)} g_s(A)
&\sim T^2 N^{k-1}(\log N)^{k \dl - \frac 1{8}} \big[ \log N - \frac 1{8(k-1)} \log \log N \big]^{\frac 12} \\
	& \sim T^2 N^{k-1}(\log N)^{k \dl + \frac 3{8}}.
\end{align*}

\noi Then, it suffices to choose $T$ satisfying
\[
N^{- \frac{k-1}2}(\log N)^{- \frac k2 \dl - \frac 3{16}}
\ll T \ll \min \big( N^{-1}, N^{- \frac{k-1}2}(\log N)^{- \frac{k-1}2 \dl + \frac 1{16}} \big)
\]

\noi which is possible since $k \geq 5$ and $\dl>0$.

{\bf Case 3:} Assume $- \frac d2 < s < \min(s_{\rm scal}, 0)$. Set
\[
R = N^{-s} A^{-\frac d2} (\log N)^{-1} \quad \text{ and } \quad A = N^\theta
\]

\noi with $0 < \theta < 1$.
Then,
by $-\frac d2 <s<0$,
(ii), (iii) and (iv) are satisfied. Besides, (i) and (v) become 
\begin{align*}
\textup{(i')} & \quad T \ll N^{-1} \text{ and } T \ll N^{s \frac{k-1}2 -\frac d4 (k-1) \theta} (\log N)^{\frac{k-1}2}, \\
\textup{(v')} & \quad T^2 N^{-ks + (s + \frac d2 (k-1)) \theta} (\log N)^{-k} \gg n.
\end{align*}

\noi Then, it suffices to choose $T$ satisfying
\[
N^{\frac{ks}2 - (s + \frac d2 (k-1)) \frac \theta2} (\log N)^{\frac k2}
\ll T \ll \min \big( N^{-1}, N^{s \frac{k-1}2 -\frac d4 (k-1) \theta} \big).
\]
Namely,
we need to choose $\theta$ with
\[
\frac{ks}2 - \Big( s + \frac d2 (k-1) \Big) \frac \theta2
< \min \Big(-1, s \frac{k-1}2 -\frac d4 (k-1) \theta \Big).
\]
Since $k \ge 2$ and $-\frac d2 < s<0$,
this condition is equivalent to
\[
\frac{ks+2}{s+ \frac d2(k-1)} < \theta < k.
\]
Hence, we can choose an appropriate $\theta \in (0,1)$ when
\[
-\frac d2< s< \min \Big( \frac d2 - \frac 2{k-1}, 0 \Big) = \min (s_{\rm scal}, 0).
\]

\noi Therefore,
we obtain norm inflation for $d \geq 1$ and $(s,k)$ satisfying the assumption \eqref{eq:condspart1}.

\section{A second result of norm inflation}
\label{Sec4}

In this section, we prove parts \ref{THM:mainB} and \ref{THM:mainC} of Theorem \ref{THM:main}. As in Section \ref{Sec3}, we prove in fact the following proposition, which implies our desired result.

\begin{proposition}
\label{PROP:NIwithILORn}
Let $ d \geq 1$, $ 2 \leq k \leq 5$, $\s \in \R$ and $s < - \frac 1k$. Fix $(u_0 , u_1) \in \mathcal{S} (\M^d)\times \mathcal{S}(\M^d)$. Then, for any $n \in \NB$, there exists a solution $u_n$ to \eqref{vNLW} and $t_n \in (0, \frac 1n )$ such that
\begin{equation}
\label{EQ:NIwithILORn}
\| (u_n (0), \dt u_n (0)) - (u_0, u_1) \|_{\H^s} < \frac 1n \quad \text{ and } \quad \| u_n (t_n ) \|_{H^\s} > n.
\end{equation}
\end{proposition}

Note that, if \eqref{EQ:NIwithILORn} is proved for $\s \leq s$, then it follows for $\s_1 > s$ by the inequality
$$
\| u_n (t_n ) \|_{H^{\s_1}} \geq \| u_n (t_n ) \|_{H^s} > n.
$$

\noi Thus, we assume throughout the rest of this section that
\begin{equation}
\notag
\s \leq s < 0.
\end{equation}

Let us now explain the outlines of this proof. The idea is to take mostly the same construction, even though the initial data $(\vec{\phi}_n)_{n \in \NB}$ will be slightly changed depending on the nonlinearity we have (see \eqref{eq:phin}). Therefore, we keep globally the same notations. The main difference is that, contrary to the assumptions in Proposition~\ref{PROP:MultilinearEstPart2}, we assume $N^{-1} \ll T \ll 1$.
See Proposition \ref{PROP:NewEst1} for example.
This allows us to have different estimates on the terms of our power series, hence giving way to another kind of norm inflation.

\begin{remark} \rm
\label{REM:ProofCompletes}
In this remark, we come back on the proof of part~\ref{THM:mainA} of Theorem~\ref{THM:main} and explain how this section completes the argument given in Section~\ref{Sec3}. First, note that when $\s = s$, \eqref{EQ:NIwithILOR} is equivalent to \eqref{EQ:NIatGID}. Consequently, this proof also implies \eqref{EQ:NIatGID} when $2 \le k \le 5$ and $s < - \frac d2$ since we always have $-\frac d2 \leq - \frac 1k$. Besides, the proof given in Section~\ref{Sec3} does not cover the quadratic nonlinearity in one and two dimensions, or the cubic nonlinearity in the one dimension. Note however that, for these equations, we have
$$
s_{\rm scal} = \frac d2 - \frac{2}{k-1} < - \frac 1k.
$$

\noi Thus, the proof given in this section cover these cases as well and completes the proof of part~\ref{THM:mainA} of Theorem~\ref{THM:main}.
\end{remark}

\subsection{A new lower bound}

Let us first come back on the notations. We take mainly the notations introduced in Section~\ref{Sec3}. However, we make the following distinction for $\vec{\phi}_n$. Set
$$
\Si_k
:=
\begin{cases}
\{ -N e_1, N e_1 \}
& \text{if $k$ is even},
\\
\{ -2N e_1, -N e_1, N e_1, 2N e_1\}
& \text{if $k$ is odd},
\end{cases}
$$

\noi and 
\begin{equation}
\O_k := \bigcup_{\eta \in \Si_k} (\eta + Q_A).
\label{Omek}
\end{equation}

\noi Define
\begin{equation}
\label{eq:phin}
\ft{\phi}_n \coloneqq 
R \ind_{\O_k}
\end{equation}

\noi and $\vec \phi_n = (\phi_n , 0)$.
Let us now move onto a new lower bound.

\begin{proposition}
\label{PROP:NewEst1}
Let
$d \ge 1$,
$2 \le k \le 5$,
$1 \le A \ll N$ and
$N^{-1} \ll t \les 1$.
Then, for any $s \in \R$, we have
\begin{equation}
\label{EQ:NewEst1}
\| \Xi_1 (\vec{\phi}_n) (t) \|_{H^s}
\ges
R^k
\frac tN A^{(k-1)d}
\min \Big( g_s(A), g_s \big( \frac 1t \big) \Big),
\end{equation}
	where $g_s(A)$ is defined in \eqref{DEF:gs}.
\end{proposition}

\begin{proof}
Again, we take the same notations as in the proof of Proposition~\ref{PROP:MultilinearEstPart2}. Recall that we have 
\begin{equation}
\notag
\begin{aligned}
\F_x \big[ \Xi_1 (\vec{\phi}_n) \big] (t, \xi)
&= -R^k \sum_{(\eta_1 , \dots, \eta_k) \in \Si^k_k} \int^t_0 e^{-\frac 12 \abs{\xi} (t-t')} V_1(t - t', \xi)
\\
&\qquad \times
\int_{\G_\xi} \prod^k_{j = 1}  e^{-\frac 12 \abs{\xi_j} t'} V_0 (t', \xi_j) \ind_{\eta_j + Q_A} (\xi_j) \mathrm{d}\xi_\G \mathrm{d}t'.
\end{aligned}
\end{equation}

\noi where $\Si^k_k \coloneqq \Si_k \times \cdots \times \Si_k$. Since for any function $f$, we have $\| f \|_{H^s ( \M^d)} \geq \| f \|_{H^s (Q_A) }$, we can also assume that $\xi \in Q_A$, which gives according to Lemma \ref{LEM:ConvolutionIneq}
\begin{equation}
\notag
\eta_1 + \cdots + \eta_k = 0.
\end{equation}

Moreover, observe that
$$
V_0 (t, \xi) = \cos\bigg(\frac{\sqrt{3}}{2} t \abs{\xi}\bigg) + \frac{1}{\sqrt{3}} \sin\bigg(\frac{\sqrt{3}}{2} t \abs{\xi}\bigg) = \frac{2}{\sqrt{3}} \cos\bigg(\frac{\sqrt{3}}{2} t \abs{\xi} - \frac{\pi}{6}\bigg)
$$
\noi and by a product-to-sum formula,
we get
$$
\prod^{k}_{j = 1} V_0 (t, \xi_j) = \sum_{\eps_1, \cdots, \eps_k \in \{ \pm 1\}} \frac{1}{\sqrt{3}^k} \cos\bigg(\frac{\sqrt{3}}{2} t \Psi(\overline{\xi}, \overline{\eps}) - S(\overline{\eps}) \frac{\pi}{6} \bigg)
$$

\noi where we denote 
\begin{align}
\overline{\xi} &:= (\xi_1, \dots, \xi_k), &
\overline{\eps} &:= (\eps_1, \dots, \eps_k),
\notag
\\
\Psi(\overline{\xi}, \overline{\eps}) &:= \sum^k_{j = 1} \eps_j \abs{\xi_j}, &
S(\overline{\eps}) &:= \sum^k_{j = 1} \eps_j.
\label{lowb1}
\end{align}

\noi Therefore, if we denote also
\begin{equation}
\Phi(\overline{\xi}) = \abs{\xi} - \sum^k_{j=1} \abs{\xi_j},
\label{lowb2}
\end{equation}
we get 
\begin{align*}
&\mathds{1}_{Q_A}(\xi)\F_x \big[ \Xi_1 (\vec{\phi}_n) \big] (t, \xi)
\\
&=
-  \frac{R^k}{\sqrt{3}^k} \sum_{\substack{(\eta_1, \dots, \eta_k) \in \Si^k \\ \eta_1 + \cdots + \eta_k = 0}}
\mathds{1}_{Q_A} (\xi) e^{- \frac 12 \abs{\xi}t }
\int_{\G_\xi} \sum_{\eps_1, \cdots, \eps_k \in \{ \pm 1\}} \\
	& \qquad \times \underbrace{\bigg(\int^t_0 e^{\frac{t'}{2}\Phi(\overline{\xi}) } \frac{\sin \big( \frac{\sqrt{3}}{2} (t-t') \abs{\xi} \big)}{ \frac{\sqrt{3}}{2} \abs{\xi}}
	\cos\bigg(\frac{\sqrt{3}}{2} t' \Psi(\overline{\xi}, \overline{\eps}) - S(\overline{\eps}) \frac{\pi}{6} \bigg) \mathrm{d}t'}_{\eqqcolon I (t, \overline{\xi})}
\\
	& \qquad \times 
\prod^k_{j=1} \mathds{1}_{\eta_j + Q_A} (\xi_j) \mathrm{d}\G_\xi. 
\end{align*}

Here,
a direct calculation shows the following indefinite integral:
\[
\int e^{a t'} \cos (bt' + c) dt'
= e^{a t'} \frac{a \cos( bt' + c ) + b \sin (bt' + c)}{a^2+b^2}
\]
for any $a,b, c \in \R$.
Then, integration by parts implies that
\begin{align}
I (t, \overline{\xi}) 
&=
-\frac{\frac 12 \Phi(\overline{\xi}) \cos(S(\overline{\eps})\frac{\pi}{6}) - \frac{\sqrt{3}}{2} \Psi(\overline{\xi}, \overline{\eps})\sin(S(\overline{\eps})\frac{\pi}{6}) }
{\frac 14 \Phi(\overline{\xi})^2 + \frac{3}{4} \Psi(\overline{\xi}, \overline{\eps})^2}
\frac{\sin \big( \frac{\sqrt{3}}{2} t \abs{\xi} \big)}{ \frac{\sqrt{3}}{2} \abs{\xi}}
\notag
\\
&\quad
+
\int_0^t
e^{\frac{t'}{2} \Phi(\overline{\xi})}
\frac{\frac 12 \Phi(\overline{\xi}) \cos ( \frac{\sqrt{3}}{2} t' \Psi(\overline{\xi}, \overline{\eps}) -  S(\overline{\eps})\frac{\pi}{6} )+ \frac{\sqrt{3}}{2} \Psi(\overline{\xi}, \overline{\eps})  \sin (  \frac{\sqrt{3}}{2} t' \Psi(\overline{\xi} \overline{\eps}) -  S(\overline{\eps})\frac{\pi}{6} )}
{\frac 14 \Phi(\overline{\xi})^2 + \frac{3}{4} \Psi(\overline{\xi}, \overline{\eps})^2}
\notag
\\&\qquad
\times
\cos \bigg( \frac{\sqrt{3}}{2} (t-t') \abs{\xi} \bigg)
dt'
\notag
\\
&=: \mathrm I (\cj \xi , \cj \eps)
\frac{\sin \big( \frac{\sqrt{3}}{2} t \abs{\xi} \big)}{ \frac{\sqrt{3}}{2} \abs{\xi}}  + \mathrm{II}(t,\cj \xi, \cj \eps).
\label{lowb3}
\end{align}

\noi
Since we assumed $\xi \in Q_A = [-\frac A8, \frac A8]^d$, and we have $1 \leq A \ll N \sim \abs{\xi_j }$ for any $j = 1, \dots , k$, we get $\Phi(\overline{\xi})  \sim - N$.
Then,
we have
\[
| \mathrm{II} (t,\cj \xi, \cj \eps )|
\les
\int_0^t
\frac{e^{\frac {t'}2 \Phi(\overline{\xi})}}{N}
dt'
\les
\frac 1{N^2}.
\]

\noi Now, Lemma \ref{LEM:ConvolutionIneq} gives
$$
\begin{aligned}
&\bigg|\mathds{1}_{Q_A}(\xi)\F_x \big[ \Xi_1 (\vec{\phi}_n) \big] (t, \xi)
\\
&\
+
\frac{R^k}{\sqrt{3}^k}
\mathds{1}_{Q_A} (\xi) e^{- \frac 12 \abs{\xi}t }
\frac{\sin \big( \frac{\sqrt{3}}{2} t \abs{\xi} \big)}{ \frac{\sqrt{3}}{2} \abs{\xi}}
\sum_{\substack{(\eta_1, \dots, \eta_k) \in \Si_k^k \\ \eta_1 + \cdots + \eta_k = 0}}
\int_{\G_\xi}
\sum_{\eps_1, \cdots, \eps_k \in \{ \pm 1\}}
\mathrm I (\cj \xi , \cj \eps)
\prod^k_{j=1} \mathds{1}_{\eta_j + Q_A} (\xi_j) \mathrm{d}\xi_\G
\bigg|
\\
&\les
R^k
\mathds{1}_{Q_A} (\xi)
e^{- \frac 12 \abs{\xi}t }
\sum_{\substack{(\eta_1, \dots, \eta_k) \in \Si_k^k \\ \eta_1 + \cdots + \eta_k = 0}}
\int_{\G_\xi} \prod^k_{j=1} \mathds{1}_{\eta_j + Q_A} (\xi_j) \mathrm{d}\xi_\G
\frac 1{N^2}
\\
&\les
R^k
A^{d(k-1)}
\mathds{1}_{Q_A} (\xi)
e^{- \frac 12 \abs{\xi}t }
\frac 1{N^2}
.
\end{aligned}
$$

\noi Besides, we have from Lemma \ref{BoundBigI} below that 
$$
\sum_{\eps_1, \cdots, \eps_k \in \{ \pm 1\}} \mathrm I (\cj \xi , \cj \eps)
\sim \frac 1N,
$$
where the implicit constant is independent of the choice of $\cj \xi$.
Hence,
with Lemma \ref{LEM:ConvolutionIneq},
we obtain that
\begin{align*}
&\| \Xi_1 (\vec{\phi}_n) (t) \|_{H^s}
\\
&\ge
\| \jb{\xi}^s \mathds{1}_{Q_A}(\xi)\F_x \big[ \Xi_1 (\vec{\phi}_n) \big] (t, \xi) \|_{L^2_\xi}
\\
&\ges
R^k
A^{d(k-1)}
\bigg(
\frac 1N
\Big\|
\jb{\xi}^s
\mathds{1}_{Q_A} (\xi) e^{- \frac 12 \abs{\xi}t }
\frac{\sin \big( \frac{\sqrt{3}}{2} t \abs{\xi} \big)}{ \frac{\sqrt{3}}{2} \abs{\xi}}
\Big\|_{L^2_\xi}
-
\frac 1{N^2}
\Big\| \jb{\xi}^s 
\mathds{1}_{Q_A} (\xi)
e^{- \frac 12 \abs{\xi}t } \Big\|_{L^2_\xi}
\bigg)
\\
&=:
R^k
A^{d(k-1)}
\Big(
\frac 1N
D_1 (t,A) - \frac 1{N^2} D_2(t,A)
\Big).
\end{align*}

Here,
when $0< t \ll A^{-1}$,
by \eqref{DEF:gs},
we have
\begin{align*}
D_1(t,A)
&\sim
t
\big\|
\jb{\xi}^s
\mathds{1}_{Q_A}
\big\|_{L^2_\xi}
\sim t g_s(A),
\\
D_2(t,A)
&\les
\| \jb{\xi}^s \mathds{1}_{Q_A} \|_{L^2_\xi}
\sim g_s(A).
\end{align*}
When $A^{-1} \les t \les 1$,
a direct calculation shows that
\begin{align*}
D_1(t,A)^2
&\sim
t^2 \int_0^1 dr
+ t^2 \int_1^{\frac 1{t}} r^{2s+d-1} dr
+ \int_{\frac 1t}^{\frac A8} r^{2s+d-3} e^{-r t} \sin^2 \Big( \frac{\sqrt 3}2 t r \Big) dr
\\
&\sim
\begin{cases}
t^2 & \text{if } s<-\frac d2, \\
t^2 \log \jb{ \frac 1t} & \text{if } s=-\frac d2, \\
t^{-2s-d+2} & \text{if } s>-\frac d2
\end{cases}
\\
&= t^2 g_s \big (\frac 1t \big)^2.
\end{align*}
Similarly,
we have
\begin{equation}
\begin{aligned}
D_2(t,A)^2
\le
\big\| \jb{\xi}^s  e^{- \frac 12 \abs{\xi}t } \big\|_{L^2}^2
&\sim
\int_0^1 dr
+ \int_1^{\frac 1{t}} r^{2s+d-1} dr
+ \int_{\frac 1t}^{\infty} r^{2s+d-1} e^{-r t} dr
\\
&\sim
\begin{cases}
1 & \text{if } s<-\frac d2, \\
\log \jb{\frac 1t} & \text{if } s=-\frac d2, \\
t^{-2s-d} & \text{if } s>-\frac d2
\end{cases}
\\
&= g_s \big (\frac 1t \big)^2.
\end{aligned}
\notag
\end{equation}
From $t \gg N^{-1}$,
we obtain the desired bound \eqref{EQ:NewEst1}.
\end{proof}

\begin{lemma}
\label{BoundBigI}
Taking the same notations as before, we have
$$
\sum_{\eps_1, \cdots, \eps_k \in \{ \pm 1\}} \mathrm I (\cj \xi , \cj \eps)
\sim \frac 1N,
$$
where the implicit constant is independent of the choice of $\cj \xi$.
\end{lemma}

\begin{proof}
It follows from \eqref{lowb3}, \eqref{lowb1} and \eqref{lowb2} that
\begin{align*}
\mathrm I (\cj \xi , \cj \eps)
&= -
\frac{ 2 \Phi(\overline{\xi}) \cos(S(\overline{\eps})\frac{\pi}{6}) - 2\sqrt{3} \Psi(\overline{\xi}, \overline{\eps})\sin(S(\overline{\eps})\frac{\pi}{6}) }{ \Phi(\overline{\xi})^2 + 3 \Psi(\overline{\xi}, \overline{\eps})^2}
\\
&=
-2
\frac{ \cos(S(\overline{\eps})\frac{\pi}{6})}{ \Phi(\overline{\xi})^2 + 3 \Psi(\overline{\xi}, \overline{\eps})^2}
|\xi|
+
4 \sum_{j=1}^k
\frac{ \sin ((1 + \eps_j S(\overline{\eps})) \frac{\pi}{6})}
{ \Phi(\overline{\xi})^2 + 3 \Psi(\overline{\xi}, \overline{\eps})^2} |\xi_j|
\\
&=:
-2\mathrm I_0 (\cj \xi , \cj \eps)
+
4
\sum_{j=1}^k
\mathrm I_j (\cj \xi , \cj \eps).
\end{align*}
Since $ \Phi(\overline{\xi})^2 + 3 \Psi(\overline{\xi}, \overline{\eps})^2 \sim N^2$ and $\abs{\xi} \les A \ll N \sim \abs{\xi_j}$ for any $j = 1, \dots , k$,
we have
\[
|\mathrm I_0 (\cj \xi , \cj \eps)|
\les \frac A{N^2}.
\]

Once we have
\begin{equation}
\sum_{\overline{\eps} \in \{ \pm 1\}^k}
\mathrm I_j (\cj \xi , \cj \eps)
\ges \frac 1N
\label{bounde1}
\end{equation}
for $j=1, \dots, k$,
we obtain
\[
\frac 1N
\sim
\frac 1N - \frac A{N^2}
\les
\sum_{\overline{\eps} \in \{ \pm 1\}^k}
\mathrm I (\cj \xi , \cj \eps)
\les \frac 1N + \frac A{N^2} \sim \frac 1N.
\]
Note that
the upper bound
$\sum_{\overline{\eps} \in \{ \pm 1\}^k}
\mathrm I_j (\cj \xi , \cj \eps)
\les \frac 1N$
easily follows from $\xi_1, \dots, \xi_k \in \O_k$.

In what follows,
we consider \eqref{bounde1}.
For simplicity,
we only focus on the case $j=1$,
since the remaining cases are similarly handled.
By \eqref{lowb1},
we have
\[
1 + \eps_1 S(\overline{\eps})
= 2 + \eps_1 \sum_{j=2}^k \eps_j.
\]
Since $0 \le 1 + \eps_1 S(\overline{\eps}) \le 4$ for $k=2,3$ and $\overline{\eps} \in \{ \pm 1\}^k$,
we obtain that
$\sin ((1 + \eps_1 S(\overline{\eps})) \frac{\pi}{6}) \ge 0$
and
\[
\sum_{\overline{\eps} \in \{ \pm 1\}^k}
\mathrm I_1 (\cj \xi , \cj \eps)
\ge \mathrm I_1 (\cj \xi , \{ +1 \}^k)
\sim
\frac 1N.
\]
for $k=2,3$.

For $k=4$,
we have
$\sin ((1 + \eps_1 S(\overline{\eps})) \frac{\pi}{6})<0$
only when
$\overline{\eps}
= \pm (+1,-1,-1,-1)$.
For $\overline{\eps} = \pm (+1,-1,-1,-1)$,
we set
\[
\overline{\eps^\ast}
=
\pm
(+1,+1,-1,-1).
\]
It follows from
\eqref{lowb1},
$\xi_1, \dots, \xi_4 \in \O_4$ and \eqref{Omek}
that
\[
|\Psi(\overline{\xi}, \overline{\eps})|
- 
|\Psi(\overline{\xi}, \overline{\eps^\ast})|
\ge 2 \min \big( |\xi_2|, -|\xi_1|+|\xi_3|+|\xi_4| \big)
\sim N.
\]
We thus obtain that
\begin{align*}
\mathrm I_1 (\cj \xi , \cj{\eps^\ast}) + \mathrm I_1 (\cj \xi , \cj \eps)
&=
\frac{1}2
|\xi_1|
\Big(
\frac{1}
{ \Phi(\overline{\xi})^2 + 3 \Psi(\overline{\xi}, \overline{\eps^\ast})^2}
-\frac{1}
{ \Phi(\overline{\xi})^2 + 3 \Psi(\overline{\xi}, \overline{\eps})^2}
\Big)
\\
&=
\frac{1}2
|\xi_1|
\frac{3 (|\Psi(\overline{\xi}, \overline{\eps})| + |\Psi(\overline{\xi}, \overline{\eps^\ast})|)(|\Psi(\overline{\xi}, \overline{\eps})| - |\Psi(\overline{\xi}, \overline{\eps^\ast})|)}
{ (\Phi(\overline{\xi})^2 + 3 \Psi(\overline{\xi}, \overline{\eps})^2) (\Phi(\overline{\xi})^2 + 3 \Psi(\overline{\xi}, \overline{\eps})^2)}
\sim
\frac 1N,
\end{align*}
which shows \eqref{bounde1} for $j=1$ and $k=4$.

For $k=5$,
we have
$\sin ((1 + \eps_j S(\overline{\eps})) \frac{\pi}{6})<0$
only when
$\overline{\eps} = \pm (+1, -1,-1,-1,-1)$.
For $\overline{\eps} = \pm (+1, -1,-1,-1,-1)$,
we set
\[
\overline{\eps^\ast}
=
\pm
(+1, +1, +1,-1,-1).
\]
It follows from
\eqref{lowb1},
$\xi_1, \dots, \xi_5 \in \O_5$ and \eqref{Omek}
that
\[
|\Psi(\overline{\xi}, \overline{\eps})|
- 
|\Psi(\overline{\xi}, \overline{\eps^\ast})|
\ge 2 \min \big( |\xi_2|+|\xi_3|, -|\xi_1|+|\xi_4|+|\xi_5| \big)
\ge 0.
\]
We have
\begin{align*}
\mathrm I_1 (\cj \xi , \cj \eps)
+
\mathrm I_1 (\cj \xi , \cj{\eps^\ast})
&=
\frac{\sqrt 3}2
|\xi_1|
\Big(
\frac{1}
{ \Phi(\overline{\xi})^2 + 3 \Psi(\overline{\xi}, \overline{\eps})^2}
-
\frac{1}
{ \Phi(\overline{\xi})^2 + 3 \Psi(\overline{\xi}, \overline{\eps^\ast})^2}
\Big)
\\
&=
\frac{\sqrt 3}2
|\xi_1|
\frac{3 (|\Psi(\overline{\xi}, \overline{\eps^\ast})| + |\Psi(\overline{\xi}, \overline{\eps})|) (|\Psi(\overline{\xi}, \overline{\eps^\ast})| - |\Psi(\overline{\xi}, \overline{\eps})|)}
{ (\Phi(\overline{\xi})^2 + 3 \Psi(\overline{\xi}, \overline{\eps^\ast})^2) (\Phi(\overline{\xi})^2 + 3 \Psi(\overline{\xi}, \overline{\eps})^2)}
\ge 0.
\end{align*}
We thus obtain
\[
\sum_{\overline{\eps} \in \{ \pm 1\}^k}
\mathrm I_1 (\cj \xi , \cj \eps)
\ge \mathrm I_1 (\cj \xi , (+1,+1,-1,+1,-1))
\sim
\frac 1N,
\]
which shows \eqref{bounde1} for $j=1$ and $k=5$.
\end{proof}

\begin{remark}
\label{REM:evenk}
\rm
For even $k$,
we can show the same bound.
Indeed,
if
$\sin ((1 + \eps_1 S(\overline{\eps})) \frac{\pi}{6}) < 0$,
there exists $\overline{\eps^\ast} \in \{ \pm 1 \}^k$ such that
\begin{equation}
-\sin \Big((1 + \eps_1 S(\overline{\eps})) \frac{\pi}{6}\Big)
= \sin \Big((1 + \eps_1^\ast S(\overline{\eps^\ast})) \frac{\pi}{6}\Big),
\quad
|S(\overline{\eps})| > |S(\overline{\eps^\ast})|.
\label{condev}
\end{equation}
Since $k$ is even,
it follows from
\eqref{lowb1},
$\xi_j \in \O_k$ and \eqref{Omek}
that
\[
|\Psi(\overline{\xi}, \overline{\eps})|
\ge |S(\overline{\eps})| N - kA,
\quad
|\Psi(\overline{\xi}, \overline{\eps^\ast})|
\le |S(\overline{\eps^\ast})| N + kA.
\]
We then obtain that
\begin{align*}
&\mathrm I_1 (\cj \xi , \cj \eps)
+ \mathrm I_1 (\cj \xi , \cj{\eps^\ast})
\\
&\ge
\Big| \sin \Big( (1 + \eps_1 S(\overline{\eps})) \frac{\pi}{6} \Big) \Big|
|\xi_1|
\Big(
\frac{1}
{ \Phi(\overline{\xi})^2 + 3 (|S(\overline{\eps^\ast})| N + kA)^2}
-
\frac{1}
{ \Phi(\overline{\xi})^2 + 3 (|S(\overline{\eps})| N - kA)^2}
\Big)
\\
&=
\Big| \sin \Big( (1 + \eps_1 S(\overline{\eps})) \frac{\pi}{6} \Big) \Big| |\xi_1|
\frac{3 (|S(\overline{\eps})| +|S(\overline{\eps^\ast})|) N
\{ (|S(\overline{\eps})|-|S(\overline{\eps^\ast})|)N -2kA \}}
{ (\Phi(\overline{\xi})^2 + 3 (|S(\overline{\eps^\ast})| N + kA)^2)
( \Phi(\overline{\xi})^2 + 3 (|S(\overline{\eps})| N + kA)^2)}
\\
&\sim \frac 1N.
\end{align*}

For odd $k$,
if we can choose $\overline{\eps^\ast} \in \{ \pm 1 \}^k$ satisfying
\eqref{condev} and
$|\Psi(\overline{\xi},\overline{\eps})| \ge |\Psi(\overline{\xi},\overline{\eps^\ast})|$,
the same argument works.
More precisely,
the condition \eqref{condev} is not sufficient for odd $k \ge 7$.
Indeed,
when
$k=7$ and $\overline{\eps} = (+1,+1,-1,-1,-1,-1,-1)$,
we have to choose
\[
\overline{\eps^\ast}
=
\begin{cases}
(+1,+1,-1,+1,+1,-1,-1)
&
\text{if }
(\xi_1,\dots, \xi_{7})
= (2N,-2N,2N,N,-N,-N,-N),\\
(+1,+1,+1,-1,+1,-1,-1)
&
\text{if }
(\xi_1,\dots, \xi_{7})
= (2N,-2N,N,2N,-N,-N,-N).
\end{cases}
\]
However, if the choice of $\overline{\eps^\ast}$ is reversed, the inequality
$|\Psi(\overline{\xi},\overline{\eps})|
\ge |\Psi(\overline{\xi},\overline{\eps^\ast})|$
does not hold,
although the condition \eqref{condev} is satisfied.
Namely,
the choice of $\overline{\eps^\ast}$
depends also on $\xi_j \in \O_k$
for odd $k \ge 7$. Therefore, we are not sure if we can generalize this method for any polynomial nonlinearity.

Nonetheless, we do not think that this issue must be explored further. Indeed, considering the results in the following subsections, a generalization of this section could still only give norm inflation with infinite loss of regularity in $H^s (\M^d)$ for $s < - \frac 1k$, with $k$ the power of the nonlinearity. In Section~\ref{Sec3} though, we already proved norm inflation in every negative Sobolev spaces whenever $k \geq 5$. Subsequently, the only advantage of such a result would only be to have some infinite loss of regularity in a certain range, which do not have any consequences on the continuity of the solution map. We decided then to leave the computations as they are.
\end{remark}

By using Proposition~\ref{PROP:NewEst1}, and following the idea of \cite{Bou97}, we can prove part~\ref{THM:mainC} of Theorem~\ref{THM:main}, that is the failure of $C^k$-continuity of our solution map in $H^s (\M^d)$ when $d \in \NB$, $2 \le k \le 5$ and $s <  s_{\rm vis}$,
where $s_{\rm vis}$ is defined in \eqref{defsWP}.

\begin{proof}[Proof of \eqref{THM:mainC} in Theorem \ref{THM:main}]

Setting $\vec u_{0,N} = \vec \phi_N$, with $\vec \phi_N$ defined by \eqref{eq:phin} and $R \sim N^{-s} A^{-\frac d2}$,
we have
$\| \vec u_{0,n} \|_{\H^s} \sim 1$. However, when $A^{-1} \les t \les 1$,
Proposition \ref{PROP:NewEst1} yields
\[
\| \Xi_1 (\vec{\phi}_n) (t) \|_{H^s}
\ges
R^k
\frac tN A^{(k-1)d} g_s \big( \frac 1t \big)
\sim
N^{-ks-1} A^{(\frac k2-1)d}
t g_s \big( \frac 1t \big).
\]

\noi Setting
\[
A = N (\log N)^{-1}
\quad \text{ and } \quad 
t = (\log N)^{-1},
\]

\noi we obtain 
\[
\| \Xi_1 (\vec{\phi}_n) (t) \|_{H^s}
\ges
N^{-ks-1+(\frac k2-1)d}
(\log N)^{-(\frac k2-1)d-1}
g_s (\log N).
\]
If
\[
-ks-1+ \Big( \frac k2-1 \Big) d>0
\iff
s< d \Big( \frac 12 - \frac 1k \Big) -\frac 1k
= s_{\rm vis}
,
\]
we have
\[
\lim_{N \to \infty}
\| \Xi_1 (\vec{\phi}_n) (t) \|_{H^s}
=\infty,
\]
which concludes the proof.
\end{proof}

\subsection{New upper bounds}
\label{subsecUB2}

In the following,
we estimate the upper bond of the iteration terms, but first we introduce some new notations and comments. We denote, for $\vec u_0 = (u_0 , u_1)$,
\begin{equation}
\label{eq:absu0}
|\ft{\vec u_0}| := | \ft{u}_0 + \ft{u}_1|.
\end{equation}

\noi Besides, we denote for any function $f$ and $n \in \NB$
\begin{equation}
\label{eq:convIt}
f^{[n]} \coloneqq \underbrace{f \ast \cdots \ast f}_{n \text{ times}}.
\end{equation}
For simplicity,
we denote the Dirac delta measure at the origin as $f^{[0]}$.

As in Section~\ref{Sec3}, since $\vec u_0$ is fixed and smooth, we assume
$$
\| \vec u_0 \|_{\H^0}\sim \| \vec u_0 \|_{\FLv^1} \sim 1.
$$

\noi Thus, for any $s < 0$ and by using Young's inequality and the algebra property of $\FL^1 (\M^d)$, we get for any $m, r \in \NB_0$, $(m,r) \neq (0,0)$ and $2 \leq k \leq 5$:
\begin{equation}
\label{estconvIDs}
\| \ind_{\O_k}^{[m]}\ast \abs{\ft{\vec u_0}}^{[r]} \|_{H^s} \leq \| \ind_{\O_k}^{[m]}\ast \abs{\ft{\vec u_0}}^{[r]} \|_{L^2} \les A^{md}.
\end{equation} 

\noi The idea in the following proofs will be to rewrite the Picard iterates so that we can use \eqref{estconvIDs} to complete the estimates.

Let us first consider the first three Picard iterates.

\begin{lemma}
\label{lem:it2t}
For any $s < 0$, $k \geq 2$, $1 \le A \ll N$ and $N^{-1} \ll t \ll 1$,
we have
\begin{align}
\|\Xi_0 (\vec u_{0,n})(t)\|_{H^s}
&\les
RN^s A^{\frac d2}
+ 1,
\label{it04a0NormHs}
\\
\|\Xi_1 (\vec u_{0,n})(t)\|_{H^s}
&\les
 \frac tN \sum_{m = 1}^k R^m A^{md} + t^2,
\label{it04a1NormHs}
\\
\|\Xi_2 (\vec u_{0,n})(t)\|_{H^s}
&\les
\bigg(\frac tN \bigg)^2 \sum^{2k-1}_{m = k+1} R^m A^{md}  + \frac{t^3}{N} \sum^{k}_{m = 1} R^m A^{md}  + t^4 .
\label{it04a2NormHs}
\end{align}
\end{lemma}

\begin{proof}
Observe that, with \eqref{estconvIDs} and \eqref{eq:phin}, it suffices to prove
\begin{align}
|\F_x \big[\Xi_0 (\vec u_{0,n})\big](t,\xi)|
&\les
e^{-\frac 12 |\xi| t}
\Big(
R \ind_{\O_k}(\xi)
+ |\ft{\vec u_0}| (\xi) \Big),
\label{it04a0}
\\
|\F_x \big[\Xi_1 (\vec u_{0,n})\big](t,\xi)|
&\les
e^{-\frac 12 |\xi| t} \frac tN
\Bigg( \sum_{m = 1}^k R^m \big( \ind_{\O_k}^{[m]} \ast \abs{\ft{\vec u_0}}^{[k-m]} \big) (\xi)
+ t^2  |\ft{\vec u_0}|^{[k]} (\xi)
\Bigg),
\label{it04a1}
\\
|\F_x \big[\Xi_2 (\vec u_{0,n})\big](t,\xi)|
&\les
e^{-\frac 12 |\xi| t}
\Bigg(
\bigg(\frac tN \bigg)^2 \sum^{2k-1}_{m = k+1} R^m \big(\ind_{\O_k}^{[m]} \ast |\ft{\vec u_0}|^{[2k - 1 - m]} \big)(\xi) \notag \\
& \qquad + \frac{t^3}{N} \sum^{k}_{m = 1} R^m \big(\ind_{\O_k}^{[m]}\ast |\ft{\vec u_0}|^{[2k - 1 - m]} \big)(\xi)  + t^4  |\ft{\vec u_0}|^{[2k - 1 ]} (\xi) \Bigg).
\label{it04a2}
\end{align}

By \eqref{EQ:DefiV} and \eqref{eq:absu0},
we have
\eqref{it04a0}.
Let us now turn to \eqref{it04a1} and denote again $\Phi(\overline{\xi})$ as in \eqref{lowb2}.
Using \eqref{EQ:DefiDuhamelOp}, \eqref{EQ:DefiW} and \eqref{it04a0}, we have 
\begin{align}
&|\F_x \big[\Xi_1 (\vec u_{0,n})\big](t,\xi)|
\notag
\\
&\les
e^{-\frac 12 |\xi| t} 
\int_0^t \int_{\G_\xi} e^{\frac 12 \Phi(\overline{\xi})  t'} \abs{V_1(t-t', \xi)} \prod^k_{r = 1} \big(R \ind_{\O_k} + \abs{\ft{\vec u_0}}\big)(\xi_r) \mathrm{d}\xi_\G \mathrm{d}t' \notag \\
	& \les e^{-\frac 12 |\xi| t}  \sum^{k}_{m = 1} \binom{k}{m} R^m \int_{\G_\xi} \int^t_0 e^{\frac 12 \Phi(\overline{\xi})  t'} \abs{V_1(t-t', \xi)} \mathrm{d}t'
	\bigg( \prod^m_{r = 1} \ind_{\O_k}(\xi_r) \bigg) \bigg( \prod^k_{r = m+1} \abs{\ft{\vec u_0}} (\xi_r) \bigg) \mathrm{d}\xi_\G  \notag \\
	& \qquad + e^{-\frac 12 |\xi| t}  \int_{\G_\xi} \int^t_0 e^{\frac 12 \Phi(\overline{\xi})  t'} \abs{V_1(t-t', \xi)} \mathrm{d}t' \prod^k_{r = 1} \abs{\ft{ \vec u_0}} (\xi_r) \mathrm{d}\xi_\G \notag\\
	& \eqqcolon e^{- \frac 12 \abs{\xi} t} \bigg( \sum^k_{m = 1} I_m (t, \xi) + I_0 (t, \xi)\bigg).
\label{eq:sumI1}
\end{align}

\noi The idea is to study $I_0$ and $I_m$, $m = 1, \dots , k$, separately.

In the case of $I_0$, observe that since $\Phi(\overline{\xi}) \leq 0$ and $\abs{V_1(t-t', \xi)} \leq t - t'$, we get 
\begin{equation}
I_0 (t, \xi) \leq \int_{\G_\xi} \int^t_0 (t - t') \mathrm{d}t' \prod^k_{r = 1} \abs{\ft{ \vec u_0}} (\xi_r) \mathrm{d}\xi_\G
\les t^2 \abs{\ft{ \vec u_0}}^{[k]} (\xi).
\label{eq:estI01}
\end{equation}

\noi
For the study of $I_m$, $m = 1 , \dots, k$, we consider two cases.

{\bf Case 1:} When $\abs{\xi} \ll N$, observe that $\abs{\xi_1} \sim N$ by \eqref{EQ:Defiphi_n}, since $m \geq 1$.
Namely,
there exists a constant $c>0$ such that $|\xi|- |\xi_1| \le -c N$.
Then, using the inequality $\abs{V_1(t-t', \xi)} \leq t - t'$ and \eqref{eq:convIt}, we get
\begin{align}
I_m (t, \xi) \ind_{\{ \abs{\xi} \ll N\}} (\xi) & \les R^m \int_{\G_\xi} \int^t_0 e^{\frac 12 (\abs{\xi} - \abs{\xi_1}) t'} (t-t') \mathrm{d}t'
\bigg( \prod^m_{r = 1} \ind_{\O_k}(\xi_r) \bigg) \bigg( \prod^k_{r = m+1} \abs{\ft{\vec u_0}} (\xi_r) \bigg) \mathrm{d}\xi_\G \notag \\
	& \les R^m \int_{\G_\xi} \int^t_0 e^{- \frac c2 N t'} (t-t') \mathrm{d}t'
	\bigg( \prod^m_{r = 1} \ind_{\O_k}(\xi_r) \bigg) \bigg( \prod^k_{r = m+1} \abs{\ft{\vec u_0}} (\xi_r) \bigg) \mathrm{d}\xi_\G \notag \\
	& \les R^m \frac tN \big(\ind_{\O_k}^{[m]} \ast \abs{\ft{ \vec u_0}}^{[k - m]} \big)(\xi).
\label{eq:estIm1a}
\end{align}

{\bf Case 2:} When $ \abs{\xi} \ges N$, let us use the inequality $\abs{V_1(t-t', \xi)} \leq \frac 1N$, $\Phi(\overline{\xi}) \leq 0$ and \eqref{eq:convIt} to get
\begin{align}
I_m (t, \xi) \ind_{\{ \abs{\xi} \ges N \}} (\xi) & \les R^m \int_{\G_\xi} \int^t_0 \frac 1N \mathrm{d}t'
\bigg( \prod^m_{r = 1} \ind_{\O_k}(\xi_r) \bigg) \bigg( \prod^k_{r = m+1} \abs{\ft{\vec u_0}} (\xi_r) \bigg) \mathrm{d}\xi_\G \notag \\
	& \les R^m \frac tN \big(\ind_{\O_k}^{[m]} \ast \abs{\ft{ \vec u_0}}^{[k - m]} \big)(\xi).
\label{eq:estIm1b}
\end{align}

\noi Then, \eqref{it04a1} follows from \eqref{eq:sumI1}, \eqref{eq:estI01}, \eqref{eq:estIm1a} and \eqref{eq:estIm1b}.

Finally, let us turn to the proof of \eqref{it04a2}. Observe that, by \eqref{EQ:DefiDuhamelOp}, \eqref{EQ:DefiW}, \eqref{it04a1}, \eqref{it04a0} and by symmetry, we have with the same notations as before
\begin{align}
|\F_x \big[\Xi_2 (\vec u_{0,n})\big](t,\xi)|
& \les e^{- \frac 12 \abs{\xi}t} \int^t_0 \int_{\G_\xi} e^{\frac 12 \Phi(\overline{\xi}) t'} \abs{V_1 (t-t', \xi)} \notag \\
	& \qquad \times  \Bigg( \frac{t'}{N} \sum_{m = 1}^k R^m \big( \ind_{\O_k}^{[m]} \ast \abs{\ft{\vec u_0}}^{[k-m]} \big) (\xi_1)
+ t'^2 \big( |\ft{\vec u_0}|^{[k]}\big) (\xi_1)
\Bigg) \notag \\
	& \qquad \times \prod^k_{r = 2} \bigg( R \ind_{\O_k} (\xi_r) + |\ft{\vec u_0}|(\xi_r) \bigg) \mathrm{d}\xi_\G \mathrm{d}t' \notag \\
	 & \eqqcolon e^{-\frac 12 \abs{\xi} t} \big(I^{(2)}_1 (t, \xi) + I^{(2)}_2 (t, \xi) + I^{(2)}_3 (t, \xi) \big)
\label{eq:splitu2}
\end{align}

\noi where we define $I^{(2)}_1$, $I^{(2)}_2$ and $I^{(2)}_3$ by
\begin{equation}
\label{eq:defI21}
\begin{aligned}
I^{(2)}_1 (t, \xi) = \int_{\G_\xi} \int^t_0   e^{\frac 12 \Phi(\overline{\xi}) t'} \abs{V_1 (t-t', \xi)}  \frac{t'}{N}  & \bigg( \sum_{m = 1}^k R^m \big( \ind_{\O_k}^{[m]} \ast \abs{\ft{\vec u_0}}^{[k-m]} \big) (\xi_1) \bigg) \\ 
	&  \times \prod^k_{r = 2} \bigg( R \ind_{\O_k} (\xi_r) + |\ft{\vec u_0}|(\xi_r) \bigg) \mathrm{d}\xi_\G \mathrm{d}t',
\end{aligned}
\end{equation}
\begin{align}
\notag
I^{(2)}_2 (t, \xi) & = \sum_{\overline{v} = (v_2, \dots, v_k) \in E} \int_{\G_\xi} \int^t_0   e^{\frac 12 \Phi(\overline{\xi}) t'} \abs{V_1 (t-t', \xi)}  t'^2 |\ft{\vec u_0}|^{[k]} (\xi_1) \prod^k_{r = 2} v_r (\xi_r) \mathrm{d}\xi_\G \mathrm{d}t' \\
	& \eqqcolon  \sum_{\overline{v} = (v_2, \dots, v_k) \in E} I^{(2)}_{2, \overline{v}} (t, \xi) \notag
\end{align}

\noi where $E$ is defined by 
\begin{equation}
\label{eq:defE}
E = \bigg\{ (v_2, \dots, v_k) \in \big\{ R\ind_{\O_k}, |\ft{\vec u_0}|\big\}^{k-1} \bigg| (v_2, \dots , v_k ) \neq ( |\ft{\vec u_0}|, \dots, |\ft{\vec u_0}| ) \bigg\}
\end{equation}

\noi and
\begin{equation}
\notag
I^{(2)}_3 (t, \xi) = \int_{\G_\xi} \int^t_0   e^{\frac 12 \Phi(\overline{\xi}) t'} \abs{V_1 (t-t', \xi)}  t'^2 |\ft{\vec u_0}|^{[k]} (\xi_1) \prod^k_{r = 2} |\ft{\vec u_0}| (\xi_r) \mathrm{d}\xi_\G \mathrm{d}t'.
\end{equation}

First, observe that in a similar way as for $I_0$ in the proof of \eqref{it04a1}, the inequalities $\abs{V_1 (t-t', \xi)} \leq t-t'$ and $\Phi(\overline{\xi}) \leq 0$  give
\begin{equation}
\label{eq:estI23}
I^{(2)}_3 (t, \xi)  \leq \int_{\G_\xi} \int^t_0 (t-t')t'^2 \mathrm{d}t'  |\ft{\vec u_0}|^{[k]} (\xi_1) \prod^k_{r = 2} |\ft{\vec u_0}| (\xi_r)\mathrm{d}\xi_\G \les t^4 |\ft{\vec u_0}|^{[2k-1]} (\xi).
\end{equation}

Now, let us turn to $I^{(2)}_2$. Let us fix some $\overline{v} \in E$. According to \eqref{eq:defE}, there are $ 1 \leq m \leq k-1$ functions $v_r$ such that $v_r = R \ind_{\O_k}$ and, by \eqref{Omek}, $\abs{\xi_{r}} \sim N$. Then, we can use the same argument as for \eqref{it04a1}, namely:
\begin{itemize}
	\item when $\abs{\xi} \ll N$, by using $\Phi(\overline{\xi}) \le -c N$ and $\abs{V_1 (t-t', \xi)} \leq t-t'$, we get after rearranging and from \eqref{eq:convIt}
\begin{align*}
&\ind_{\{\abs{\xi} \ll N\}}I^{(2)}_{2, \overline{v}} (t, \xi)
\\
&\les R^m \int_{\G_\xi} \int^t_0 e^{- \frac c2 Nt'} (t-t') t'^2 \mathrm{d}t' |\ft{\vec u_0}|^{[k]} (\xi_1) 
\bigg( \prod^{m+1}_{r = 2} \ind_{\O_k} (\xi_r) \bigg) \bigg( \prod^k_{r = m+2}|\ft{\vec u_0}| (\xi_r) \bigg) \mathrm{d}\xi_\G \\
	& \les R^m \frac{t^3}{N} \big(\ind_{\O_k}^{[m]} \ast |\ft{\vec u_0}|^{[2k - 1 - m]} \big) (\xi),
\end{align*}
	\item when $\abs{\xi} \ges N$, by using $\Phi(\overline{\xi}) \leq 0$ and $\abs{V_1 (t-t', \xi)} \leq \frac 1N$, we get after rearranging and from \eqref{eq:convIt}
\begin{align*}
\ind_{\{\abs{\xi} \ges N\}}I^{(2)}_{2, \overline{v}} (t, \xi)  & \les R^m \int_{\G_\xi} \int^t_0 \frac 1N t'^2  \mathrm{d}t' |\ft{\vec u_0}|^{[k]} (\xi_1)
\bigg( \prod^{m+1}_{r = 2} \ind_{\O_k} (\xi_r) \bigg) \bigg( \prod^k_{r = m+2}|\ft{\vec u_0}| (\xi_r) \bigg) \mathrm{d}  \xi_\G \\
	& \les R^m \frac{t^3}{N} \big(\ind_{\O_k}^{[m]} \ast |\ft{\vec u_0}|^{[2k - 1 - m]} \big) (\xi).
\end{align*}
\end{itemize}

\noi Therefore, we get
\begin{equation}
\label{eq:estI22}
I^{(2)}_{2} (t, \xi) \les  \frac{t^3}{N} \sum^{k-1}_{m = 1} R^m \big(\ind_{\O_k}^{[m]} \ast |\ft{\vec u_0}|^{[2k - 1 - m]} \big)(\xi).
\end{equation}

On the other hand, observe that we can rearrange \eqref{eq:defI21} to get 
\begin{equation}
\notag
\begin{aligned}
I^{(2)}_1 (t, \xi)
&= \sum^{2k-1}_{m = 1} R^m \int_{\G_\xi} \int^t_0  e^{\frac 12 \Phi(\overline{\xi}) t'} \abs{V_1 (t-t', \xi)}  \frac{t'}{N}
\\&
\quad
\times
\bigg[ \sum_{r = 1}^k C_{m,r} \big( \ind_{\O_k}^{[r]} \ast \abs{\ft{\vec u_0}}^{[k-r]} \big) (\xi_1) 
	\bigg( \prod^{m-r + 1}_{s = 2} \ind_{\O_k} (\xi_s) \bigg) \bigg( \prod_{\widetilde{s} = m-r +2 }^k |\ft{\vec u_0}|(\xi_{\widetilde{s}}) \bigg)  \bigg] \mathrm{d}\xi_\G \mathrm{d}t',
\end{aligned}
\end{equation}

\noi where $C_{m, r} \geq 0$ are constants such that $C_{m,r} = 0$ if $r > m$ and we use the convention $\prod^{m-r + 1}_{s = 2} \ind_{\O_k} (\xi_s) = 1$ if $r = m$. We split the first sum into two parts that we denote $S_1$ and $S_2$:
\begin{equation}
\notag
\sum^{2k-1}_{m = 1} = \sum^{2k-1}_{m = k + 1} + \sum^{k}_{m = 1} \implies I^{(2)}_1 = S_1 + S_2.
\end{equation}

\noi In the case of $S_2$, by taking $\abs{V_1 (t-t', \xi)} \leq t-t'$, $e^{\frac 12 \Phi(\overline{\xi})t'} \leq 1$ as before and by applying the commutativity of convolution, we get 
\begin{equation}
\label{eq:estS2}
S_2 (t, \xi) \les \sum^k_{m=1} R^m \frac{t^3}{N} \big( \ind_{\O_k}^{[m]} \ast \abs{\ft{\vec u_0}}^{[2k-1-m]} \big) (\xi).
\end{equation}

\noi In the case of $S_1$, let us denote $S_1 = \sum^{2k-1}_{m = k+1} I^{(2)}_{1,m}$ and fix $k + 1 \leq m \leq 2k -1$. Since $m \geq k+1$ and $r \leq k$, we know from \eqref{Omek} that $\abs{\xi_2} \sim N$. Therefore, we can use the same argument as for \eqref{eq:estI22}, namely:
\begin{itemize}
	\item when $\abs{\xi} \ll N$, by taking $\abs{V_1 (t-t', \xi)} \leq t-t'$, we get from $\Phi(\overline{\xi}) \le -cN$ and \eqref{eq:convIt} that
\begin{align*}
\ind_{\{\abs{\xi} \ll N\}} I^{(2)}_{1,m} (t, \xi) & \les \frac{R^m}{N} \int_{\G_\xi} \int^t_0 e^{- \frac c2 Nt'}(t-t')t' \mathrm{d}t'  \bigg[ \sum_{r = 1}^k C_{m,r} \big( \ind_{\O_k}^{[r]} \ast \abs{\ft{\vec u_0}}^{[k-r]} \big) (\xi_1) \\ 
	&  \qquad \times
	\bigg( \prod^{m-r + 1}_{s = 2} \ind_{\O_k} (\xi_s) \bigg) \bigg( \prod_{\widetilde{s} = m-r +2 }^k |\ft{\vec u_0}|(\xi_{\widetilde{s}}) \bigg) \bigg] \mathrm{d}\xi_\G \\
	& \les R^m \bigg( \frac tN \bigg)^2 \big( \ind_{\O_k}^{[m]} \ast \abs{\ft{\vec u_0}}^{[2k - 1 -m]} \big) (\xi),
\end{align*}
	\item when $\abs{\xi} \ges N$, by taking $\abs{V_1 (t-t', \xi)} \leq \frac 1N$, we get from $\Phi(\overline{\xi}) \leq 0$ and \eqref{eq:convIt} that
\begin{align*}
\ind_{\{\abs{\xi} \ges N\}} I^{(2)}_{1,m} (t, \xi) & \les \frac{R^m}{N} \int_{\G_\xi} \int^t_0 \frac 1N t' \mathrm{d}t'  \bigg[ \sum_{r = 1}^k C_{m,r} \big( \ind_{\O_k}^{[r]} \ast \abs{\ft{\vec u_0}}^{[k-r]} \big) (\xi_1) \\ 
	&  \qquad \times
	\bigg( \prod^{m-r + 1}_{s = 2} \ind_{\O_k} (\xi_s) \bigg) \bigg( \prod_{\widetilde{s} = m-r +2 }^k |\ft{\vec u_0}|(\xi_{\widetilde{s}}) \bigg) \bigg] \mathrm{d}\xi_\G \\
	& \les R^m \bigg( \frac tN \bigg)^2 \big( \ind_{\O_k}^{[m]} \ast \abs{\ft{\vec u_0}}^{[2k - 1 -m]} \big) (\xi).
\end{align*}
\end{itemize}

\noi Thus, we get 
\begin{equation}
\label{eq:estS1}
S_1 (t, \xi) \les \sum^{2k-1}_{m = k+1} R^m \bigg( \frac tN \bigg)^2 \big( \ind_{\O_k}^{[m]} \ast \abs{\ft{\vec u_0}}^{[2k - 1 -m]} \big) (\xi)
\end{equation}

\noi and from \eqref{eq:estS1} and \eqref{eq:estS2}, we get
\begin{equation}
\label{eq:estI21}
\begin{aligned}
I^{(2)}_1 (t, \xi)
&\les
\sum^{2k-1}_{m = k+1} R^m \bigg( \frac tN \bigg)^2 \big( \ind_{\O_k}^{[m]} \ast \abs{\ft{\vec u_0}}^{[2k - 1 -m]} \big) (\xi)
\\
&\quad
+
\sum^{k}_{m = 1} R^m \frac{t^3}N \big( \ind_{\O_k}^{[m]} \ast \abs{\ft{\vec u_0}}^{[2k - 1 -m]} \big) (\xi)
.
\end{aligned}
\end{equation}

\noi Finally, \eqref{it04a2} follows from \eqref{eq:splitu2}, \eqref{eq:estI23}, \eqref{eq:estI22} and \eqref{eq:estI21}.
\end{proof}

Observe that, by using the same idea as for the proof of \eqref{it04a1NormHs}, we have, for any $s < 0$ and $N^{-1} \ll t \ll 1$,
\begin{equation}
\label{eq:normdiffXi1}
\| \Xi_1 (\vec u_{0,n})(t) - \Xi_1 (\vec \phi_{n})(t)\|_{H^s} \les \frac tN \sum_{m = 1}^{k-1} R^m + t^2.
\end{equation}

We prove another iterative bound:

\begin{lemma}
\label{lem:ujit1}
For any $s < 0$, $k \geq 2$, $j \in \NB_0$ and $N^{-1} \ll t \ll 1$,
we have
\begin{equation}
\| \Xi_j (\vec u_{0,n})(t)\|_{H^s}
\les C^{j}
t^{2j}
\bigg(
\sum_{m=1}^{(k-1)j + 1} R^{m}A^{md} (tN)^{- \frac{m}k} + 1
\bigg)
.
\notag
\end{equation}
\end{lemma}

\begin{proof}

Again, using \eqref{estconvIDs}, it suffices to prove
there exists $C= C(d,k)>0$ such that
\begin{equation}
\begin{aligned}
|\F_x & \big[\Xi_j (\vec u_{0,n})\big](t,\xi)|
\\
&\le
\frac{C^{kj+1}}{(j+1)^2}
t^{2j}
\bigg(
\sum_{m=1}^{(k-1)j + 1}
R^{m}
(tN)^{- \frac{m}k}
\Big(
\ind_{\O_k}^{[m]}
\ast
\abs{\ft{\vec u_0}}^{[(k - 1)j + 1 -m]}
\Big)
(\xi)
\\
&\hspace*{100pt}
+
\abs{\ft{\vec u_0}}^{[(k - 1)j + 1]} (\xi)
\bigg)
.
\end{aligned}
\label{it0jab}
\end{equation}

We use the induction argument.
Since $tN \gg 1$,
\eqref{it04a0} yields
$$
\abs{\F_x [\Xi_0 (\vec{u}_{0,n})] (t, \xi)}
\les e^{- \frac 12 \abs{\xi} t} \big( R \ind_{\O_k} (\xi) + \abs{\ft{\vec u_0}} (\xi) \big)
\les  R \frac{1}{(tN)^{\frac 1k}} \ind_{\O_k} (\xi) + \abs{\ft{\vec u_0}}  (\xi) .
$$
Moreover,
\eqref{it04a1} in
Lemma \ref{lem:it2t}
is a better estimate than
\eqref{it0jab} with $j=1$,
since $tN \gg 1$.

Assume that \eqref{it0jab} holds for up to $j-1$ with $j \ge 2$.
By \eqref{EQ:DefiDuhamelOp}, the induction hypothesis, and $\abs{W(t-t', \xi)} \leq t-t'$,
we have
\begin{align}
|\F_x & \big[\Xi_j (\vec u_{0,n})\big](t,\xi)|
\notag
\\
&\le
\sum_{\substack{ l_1, \dots, l_k \in \NB_0 \\ l_1+ \cdots + l_k=j-1}}
\int_0^t
\int_{\G_\xi}
\big| W(t - t',\xi)
\big|
\prod^k_{p=1} |\F_x \big[\Xi_{l_p} (\vec u_{0,n})\big](t',\xi_p)|
\mathrm{d}\xi_\G \mathrm{d}t' \notag \\
&\le
\sum_{\substack{l_1, \dots, l_k \in \NB_0 \\ l_1+ \cdots + l_k=j-1}}
\int_{\G_\xi}
\int_0^t
(t-t') \notag \\
	& \qquad \times 
\prod^k_{p = 1}
\Bigg(
\frac{C^{kl_p+1}}{(l_p+1)^2}
\sum_{m_p=1}^{(k-1)l_p + 1}
R^{m_p}
t'^{2l_p}
(t'N)^{-\frac{m_p}k}
\Big(
\ind_{\O_k}^{[m_p]}
\ast
\abs{\ft{\vec u_0}}^{[(k - 1)l_p + 1 -m_p]}
\Big)
(\xi_p)
\notag \\
& \qquad \qquad
+
\frac{C^{kl_p+1}}{(l_p+1)^2}
t'^{2l_p}
\abs{\ft{\vec u_0}}^{[(k - 1)l_p + 1]} (\xi_p)
\Bigg)
\mathrm{d}\xi_\G \mathrm{d}t' \notag \\
	& \eqqcolon
	\sum_{\substack{l_1, \dots, l_k \in \NB_0 \\ l_1+ \cdots + l_k=j-1}} \I_{\overline{l}}^{(j)} (t, \xi)
\label{eq:splituj}
\end{align}

\noi where $\overline{l} \coloneqq (l_1, \dots, l_k)$. Let us proceed to some notations. First let us denote, for any $l_p$,
\begin{align}
\text{sum}_{l_p} (t', \xi_p) & \coloneqq
\frac{C^{kl_p+1}}{(l_p+1)^2}
\sum_{m_p=1}^{(k-1)l_p + 1}
R^{m_p}
t'^{2l_p}
(t'N)^{-\frac{m_p}k}
\Big(
\ind_{\O_k}^{[m_p]}
\ast
\abs{\ft{\vec u_0}}^{[(k - 1)l_p + 1 -m_p]}
\Big) (\xi_p),
\label{eq:defsumlp} \\
\text{(ID)}_{l_p} (t', \xi_p) & \coloneqq
\frac{C^{kl_p+1}}{(l_p+1)^2}
t'^{2l_p}
\abs{\ft{\vec u_0}}^{[(k - 1)l_p + 1]} (\xi_p).
\label{eq:defIDlp}
\end{align}

\noi In addition, we expand $\I_{\overline{l}}^{(j)}$ and denote it this way
\begin{align}
\I_{\overline{l}}^{(j)} (t, \xi)
& = \int^t_0 \int_{\G_\xi} (t-t') \prod^k_{p = 1} \text{(ID)}_{l_p} (t', \xi_p) \mathrm{d}\xi_\G \mathrm{d}t'
\notag
\\
&\qquad
+ \bigg(\I_{\overline{l}}^{(j)} (t, \xi) - \int^t_0 \int_{\G_\xi} (t-t') \prod^k_{p = 1} \text{(ID)}_{l_p} (t', \xi_p) \bigg) \mathrm{d}\xi_\G \mathrm{d}t' \notag \\
	& \eqqcolon I^{(j)}_{\overline{l},0} (t, \xi) + I^{(j)}_{\overline{l},1} (t, \xi).
\label{eq:splitIcalj}
\end{align}

First, observe that according to \eqref{eq:defIDlp}
and $l_1+ \dots + l_k = j-1$, we have
\begin{align}
I^{(j)}_{\overline{l},0} (t, \xi)
&\leq
\bigg(\prod^k_{p=1} 
\frac{C^{kl_p+1}}{(l_p+1)^2} \bigg) \int^t_0 (t - t')t'^{2(j-1)} \mathrm{d}t' \abs{\ft{\vec u_0}}^{[(k - 1)(j-1) + k]} (\xi)
\notag
\\
&\les
\frac 1{(j+1)^2}
\bigg(\prod^k_{p=1} 
\frac{C^{kl_p+1}}{(l_p+1)^2} \bigg)
t^{2j} \abs{\ft{\vec u_0}}^{[(k - 1)j + 1]} (\xi).
\label{eq:estIj0}
\end{align}

\noi
On the other hand, observe that we can rewrite $I^{(j)}_{\overline{l},1}$ in the following way
\begin{align}
I^{(j)}_{\overline{l},1} (t, \xi) & = \sum_{\overline{v} = (v_{l_1}, \dots , v_{l_k}) \in \wt E}  \int_{\G_\xi} \int^t_0 (t-t') \prod^k_{p = 1} v_{l_p} (t', \xi_p) \mathrm{d}\xi_\G \mathrm{d}t' \notag \\
	& \eqqcolon \sum_{\overline{v} = (v_{l_1}, \dots , v_{l_k}) \in \wt E}  I^{(j)}_{\overline{l}, 1, \overline{v}} (t, \xi)
\label{eq:splitIj1}
\end{align}

\noi where $\wt E$ is defined by
\begin{equation}
\notag
\wt E = \Big\{ (v_{l_1}, \dots, v_{l_k})
\mid
v_{l_p} \in \big\{ \text{sum}_{l_p}, \text{(ID)}_{l_p} \big\},
\
(v_{l_1}, \dots , v_{l_k} ) \neq \big( \text{(ID)}_{l_1}, \dots, \text{(ID)}_{l_k} \big) \Big\}.
\end{equation}

\noi Observe then that, if we fix a term $ I^{(j)}_{\overline{l}, 1, \overline{v}}$, there are $1 \leq p \leq k$ functions $v_{l_r}$ such that $v_{l_r} = \text{sum}_{l_r}$. Without loss of generality, we can assume that our term is of the form
$$
I^{(j), \ast}_{\overline{l}, 1, \overline{v}} (t, \xi) = \int_{\G_\xi} \int^t_0 (t-t')
\bigg( \prod^p_{r = 1} \text{sum}_{l_r} (t', \xi_r) \bigg) \bigg( \prod^k_{r = p + 1} \text{(ID)}_{l_r} (t', \xi_r) \bigg) \mathrm{d}\xi_\G \mathrm{d}t'.
$$

\noi Then, \eqref{eq:defsumlp}, \eqref{eq:defIDlp} and the fact that $l_1 + \cdots + l_k = j-1$ give
\begin{align*}
&I^{(j), \ast}_{\overline{l}, 1, \overline{v}} (t, \xi)
\\
&= \bigg(\prod^k_{p=1}\frac{C^{kl_p+1}}{(l_p+1)^2} \bigg) \int_{\G_\xi} \int^t_0 (t-t') \sum^{(k-1)l_1 + 1}_{m_1 = 1} \cdots \sum^{(k-1)l_p + 1}_{m_p = 1}\Big[ R (t'N)^{- \frac{1}{k}} \Big]^{m_1 + \cdots + m_p}
 t'^{2(l_1 + \dots + l_p)}
\\
	& \qquad \times
	\bigg( \prod^p_{r = 1} \Big( \ind_{\O_k}^{[m_r]} \ast \abs{\ft{\vec u_0}}^{[(k - 1)l_r + 1 -m_r]} \Big) (\xi_r) \bigg)
	\bigg( \prod^k_{s = p+1} \Big(t'^{2l_s} \abs{\ft{\vec u_0}}^{[(k - 1)l_s + 1} \Big) (\xi_s) \bigg) \mathrm{d}\xi_\G \mathrm{d}t' \\
	& = \bigg(\prod^k_{p=1} \frac{C^{kl_p+1}}{(l_p+1)^2} \bigg) \sum^{(k-1)l_1 + 1}_{m_1 = 1} \cdots \sum^{(k-1)l_p + 1}_{m_p = 1} \bigg(\frac{R}{N^{\frac 1k}}\bigg)^{m_1 + \cdots + m_p} \int^t_0 (t-t')t'^{2(j-1) - \frac{m_1 + \cdots + m_p}{k}} \mathrm{d}t' \\
	& \qquad \times \int_{\G_\xi} \bigg( \prod^p_{r = 1} \Big( \ind_{\O_k}^{[m_r]} \ast \abs{\ft{\vec u_0}}^{[(k - 1)l_r + 1 -m_r]} \Big) (\xi_r) \bigg) \bigg( \prod^k_{s = p+1} \abs{\ft{\vec u_0}}^{[(k - 1)l_s + 1} (\xi_s) \bigg) \mathrm{d}\xi_\G.
\end{align*}

\noi Observe that, after rearranging the product, we have
\begin{align*}
\int_{\G_\xi} \bigg( \prod^p_{r = 1} & \Big( \ind_{\O_k}^{[m_r]} \ast \abs{\ft{\vec u_0}}^{[(k - 1)l_r + 1 -m_r]} \Big) (\xi_r) \bigg)
\bigg( \prod^k_{s = p+1} \Big( \abs{\ft{\vec u_0}}^{[(k - 1)l_s + 1} \Big) (\xi_s) \bigg) \mathrm{d}\xi_\G \\
	& = \ind_{\O_k}^{[m_1 + \cdots + m_p ]} \ast \abs{\ft{\vec u_0}}^{[(k - 1)j + 1 -m_1 - \cdots - m_p]} (\xi).
\end{align*}

\noi Besides, observe that $m_1 + \cdots + m_p \leq (k-1)(j-1) + k = (k-1)j + 1$.
By $j \ge 2$, the time integral is well defined and we have 
\begin{align*}
\int^t_0 (t-t')t'^{2(j-1) - \frac{m_1 + \cdots + m_p}{k}} \mathrm{d}t'
& = \frac 1{(2j-1 - \frac{m_1 + \cdots + m_p}{k})(2j - \frac{m_1 + \cdots + m_p}{k})}
t^{2j - \frac{m_1 + \cdots + m_p}{k}}
\\
&\le
\Big( \frac k{k+1} \Big)^2
\frac 1{(j-1)(j-\frac 1{k+1})}
t^{2j - \frac{m_1 + \cdots + m_p}{k}} \\
 & \le
\Big( \frac k{k+1} \Big)^2
\frac{2(k+1)}{(j+1)^2}
t^{2j - \frac{m_1 + \cdots + m_p}{k}}.
\end{align*}

\noi On the other hand, note that we can rewrite the sum
\begin{align*}
\sum_{m_1 = 1}^{(k-1)l_1+1} \dots \sum_{m_p = 1}^{(k-1)l_p+1} 1 & = \sum_{m_2 = 1}^{(k-1)l_1+1} \dots \sum_{m_p = 1}^{(k-1)l_p+1} \sum_{m=m_2+\dots+m_p+1}^{S_p(l_1,\dots,l_p)} 1 \\
	& \le \Big( \prod_{j=2}^p ((k-1)l_j+1) \Big)
\sum_{m=1}^{S_p(l_1,\dots,l_p)}1
\end{align*}

\noi where we denote
$m = \sum_{r=1}^p m_r$ and
$S_p(l_1, \dots, l_p) = \sum^p_{r = 1} ( (k-1)l_r + 1)$. Thus, we get
\begin{equation}
\label{eq:estIj1v}
\begin{aligned}
I^{(j),\ast}_{\overline{l}, 1, \overline{v}} (t, \xi)
& \les
\frac{1}{(j+1)^2} \frac{C^{kl_1+1}}{(l_1+1)^2} \bigg(\prod^k_{p=2} \frac{C^{kl_p+1}}{l_p+1} \bigg) \\
	& \qquad \times
\sum^{S_p(l_1, \dots, l_p)}_{m=1} R^m t^{2j} (tN)^{- \frac mk} \Big( \ind_{\O_k}^{[m]} \ast \abs{\ft{\vec u_0}}^{[(k - 1)j + 1 -m]} \Big)(\xi).
\end{aligned}
\end{equation}

\noi Observe then that $S_k(l_1, \dots, l_k) = (k-1)j + 1$. Therefore, \eqref{eq:splitIj1} and \eqref{eq:estIj1v} give, after rearranging the sums (note that we have at most $2^k$ possible $I^{(j), \ast}_{\overline{l}, 1, \overline{v}}$ terms possible)
\begin{equation}
\label{eq:estIj1}
\begin{aligned}
I^{(j)}_{\overline{l}, 1} (t, \xi) & \les
\frac{1}{(j+1)^2} \frac{C^{kl_1+1}}{(l_1+1)^2} \bigg(\prod^k_{p=2} \frac{C^{kl_p+1}}{l_p+1} \bigg) \\
	& \qquad \times
 \sum^{(k-1)j + 1}_{m=1} R^m t^{2j} (tN)^{- \frac mk} \Big( \ind_{\O_k}^{[m]} \ast \abs{\ft{\vec u_0}}^{[(k - 1)j + 1 -m]} \Big)(\xi)
\end{aligned}
\end{equation}

\noi where the implicit constant depends only on $d$ and $k$.

It follows
from \eqref{eq:splituj}, \eqref{eq:splitIcalj}, \eqref{eq:estIj0} and \eqref{eq:estIj1}
that
\begin{align*}
|\F_x & \big[\Xi_j (\vec u_{0,n})\big](t,\xi)|
\\
&\les
\frac {1}{(j+1)^2}
\sum_{\substack{l_1, \dots, l_k \in \NB_0 \\ l_1+ \cdots + l_k=j-1}} \frac{C^{kl_1+1}}{(l_1+1)^2} \bigg(\prod^k_{p=2} \frac{C^{kl_p+1}}{l_p+1} \bigg)
 \\
& \qquad
\times
\Bigg(
t^{2j} \abs{\ft{\vec u_0}}^{[(k - 1)j + 1]}
+
\sum^{(k-1)j + 1}_{m=1} R^m t^{2j} (tN)^{- \frac mk} \Big( \ind_{\O_k}^{[m]} \ast \abs{\ft{\vec u_0}}^{[(k - 1)j + 1 -m]} \Big)(\xi)
\Bigg)
\\
&\les
\frac{C^{kj} }{(j+1)^2}
\sum_{\substack{l_1, \dots, l_k \in \NB_0 \\ l_1+ \cdots + l_k=j-1}}
\frac{1}{(l_1+1)^2} \bigg(\prod^k_{p=2} \frac{1}{l_p+1} \bigg)
 \\
& \qquad
\times
\Bigg(
t^{2j} \abs{\ft{\vec u_0}}^{[(k - 1)j + 1]}
+
\sum^{(k-1)j + 1}_{m=1} R^m t^{2j} (tN)^{- \frac mk} \Big( \ind_{\O_k}^{[m]} \ast \abs{\ft{\vec u_0}}^{[(k - 1)j + 1 -m]} \Big)(\xi)
\Bigg).
\end{align*}
Here,
a direct calculation shows that
\begin{align*}
\sum_{\substack{l_1, \dots, l_k \in \NB_0 \\ l_1+ \cdots + l_k=j-1}} &
\frac{1}{(l_1+1)^2} \bigg(\prod^k_{p=2} \frac{1}{l_p+1} \bigg)
\\
&=
\sum_{l_1=0}^{j-1}
\sum_{l_2=0}^{j-1-l_1}
\dots
\sum_{l_{k-1}=0}^{j-1-l_1-\dots - l_{k-2}}
\frac{1}{(l_1+1)^2} \bigg(\prod^{k-1}_{p=2} \frac{1}{l_p+1} \bigg) 
\frac 1{j-l_1- \dots - l_{k-1}}
\\
&=
\sum_{l_1=0}^{j-1}
\sum_{l_2=0}^{j-1-l_1}
\dots
\sum_{l_{k-1}=0}^{j-1-l_1-\dots - l_{k-2}}
\frac{1}{(l_1+1)^2} \bigg(\prod^{k-2}_{p=2} \frac{1}{l_p+1} \bigg)
\\
&\qquad
\times
\frac 1{j+1-l_1- \dots - l_{k-2}}
\Big(
\frac 1{l_{k-1}+1}
+
\frac 1{j-l_1- \dots - l_{k-1}} \Big)
\\
&\les
\sum_{l_1=0}^{j-1}
\sum_{l_2=0}^{j-1-l_1}
\dots
\sum_{l_{k-2}=0}^{j-1-l_1-\dots - l_{k-3}}
\frac{1}{(l_1+1)^2} \bigg(\prod^{k-2}_{p=2} \frac{1}{l_p+1} \bigg)
\frac{\log (j+1)}{j+1-l_1- \dots - l_{k-2}}
\\
&\les \cdots \\
&\les
\sum_{l_1=0}^{j-1}
\frac{1}{(l_1+1)^2}
\frac{(\log (j+k))^{k-2}}{j+k-2-l_1}
\les 1.
\end{align*}
Hence,
we obtain that
\begin{align*}
|\F_x & \big[\Xi_j (\vec u_{0,n})\big](t,\xi)|
\\
&\les
\frac{C^{kj+1}}{(j+1)^2}
\Bigg(
t^{2j} \abs{\ft{\vec u_0}}^{[(k - 1)j + 1]}
+
\sum^{(k-1)j + 1}_{m=1} R^m t^{2j} (tN)^{- \frac mk} \Big( \ind_{\O_k}^{[m]} \ast \abs{\ft{\vec u_0}}^{[(k - 1)j + 1 -m]} \Big)(\xi)
\Bigg),
\end{align*}
which shows \eqref{it0jab}.
\end{proof}

\subsection{Proof of Proposition~\ref{PROP:NIwithILORn}}
\label{Subsec4.2}

In this subsection, we give a proof of Proposition~\ref{PROP:NIwithILORn}. Globally, the argument follows the one in the proof of Proposition~\ref{PROP:NIatGIDn}, but with the estimates proved in the previous subsection. Throughout this section, we assume that
$$
A = 1.
$$

\noi Observe then that the lefthand side of \eqref{EQ:NIwithILORn} is equivalent to
\begin{equation}
\label{eq:condID}
RN^s \ll \frac 1n.
\end{equation}

\noi As in Section~\ref{Sec3}, we want to show that the power series exists, and that $\| \Xi_1 (\vec u_{0,n}) (T) \|_{H^\s} > n $ while all the other terms are smaller. First, we get from \eqref{EQ:NewEst1} and \eqref{eq:normdiffXi1} when $N^{-1} \ll T\ll 1$
$$
\| \Xi_1 (\vec u_{0,n}) (T) \|_{H^\s} \ges R^k \frac TN -
\bigg( \frac TN \sum^{k-1}_{m = 1} R^m + T^2 \bigg).
$$

\noi Thus, if we assume
\begin{equation}
\label{cond2}
T \ll R^k N^{-1}
\end{equation}

\noi we have
\begin{equation}
\notag
\|\Xi_1 (\vec u_{0,n}) (T) \|_{H^\s} \ges R^k \frac TN
\end{equation}

\noi and we want
\begin{equation}
\label{cond3}
R^k \frac TN \gg n.
\end{equation}

\noi Note also that we have from Lemma~\ref{lem:ujit1} and \eqref{cond2}, for any $j \geq 3$,
\begin{align*}
\| \Xi_j (\vec u_{0,n}) (T) \|_{H^\s}
\les
T^{2j} \Big( R (TN)^{-\frac 1k} \Big)^{(k-1)j+1}
= \Big(R^{k-1} T^2 (TN)^{- \frac{k-1}{k}} \Big)^j R(TN)^{-\frac 1k}.
\end{align*}

\noi Let us assume that
\begin{equation}
\label{cond5}
R^{k-1} T^2 (TN)^{- \frac{k-1}{k}} \ll 1 \iff T \ll R^{- \frac{k(k-1)}{k+1}} N^{\frac{k-1}{k+1}},
\end{equation}

\noi then the series $\sum_{j \geq 0} \Xi_j (\vec u_{0,n}) (T)$ converges absolutely and we have
\begin{equation}
\notag
\sum_{j \geq 3} \|\Xi_j (\vec u_{0,n}) (T) \|_{H^\s} \les R^{3k-2} T^{3 + \frac 2k} N^{-3 + \frac 2k}.
\end{equation}

\noi Note also that, from \eqref{it04a0NormHs} and the assumption \eqref{eq:condID}, we have for any $\s \leq s < 0$
$$
\|\Xi_0 (\vec u_{0,n})\|_{H^\s}
\le
\|\Xi_0 (\vec u_{0,n})\|_{H^s} \les 1.
$$

\noi Then, we have from our considerations and \eqref{it04a2NormHs} that $u(T) = \sum_{j \geq 0} \Xi_j (\vec u_{0,n}) (T)$ exists, is a solution to \eqref{vNLW} and 
\begin{align*}
\|u(T)\|_{H^\s}
&\geq \| \Xi_1 (\vec u_{0,n}) (T) \|_{H^\s} - \| \Xi_0 (\vec u_{0,n}) (T) \|_{H^\s} - \| \Xi_2 (\vec u_{0,n}) (T) \|_{H^\s} - \sum^\infty_{j = 3} \| \Xi_j (\vec u_{0,n}) (T) \|_{H^\s} 
\\
&\ges R^k \frac TN - 1 -
\Bigg(
\bigg(\frac TN \bigg)^2 \sum^{2k-1}_{m = k+1} R^m + \frac{T^3}{N} \sum^{k}_{m = 1} R^m + T^4
\Bigg)
 - R^{3k-2} T^{3 + \frac 2k} N^{-3 + \frac 2k}.
\end{align*}

\noi Thus, we need
\begin{itemize}
	\item $R^k \frac TN \gg 1$ which is implied by \eqref{cond3},
	\item $R^k \frac TN \gg \frac{T^2}{N^2} R^m$ for any $k+1 \leq m \leq 2k-1$, which is equivalent to $R^k \gg \frac TN R^{2k-1}$,
	\item $R^k \frac TN \gg \frac{T^3}{N} R^m $ for any $1 \leq m \leq k$, and this is equivalent to $1 \gg T^2$ which is already assumed for Lemma~\ref{lem:ujit1},
	\item $R^k \frac TN \gg T^4$ which is equivalent to $ R^k N^{-1} \gg T^3$,
	\item $R^k \frac TN \gg R^{3k-2} T^{3 + \frac 2k} N^{-3 + \frac 2k} $ which is equivalent to \eqref{cond5}.
\end{itemize}

Therefore,
for the norm inflation,
we want to choose $R$ and $T$ as follows:
\begin{align*}
\textup{(i)} & \quad RN^s \ll \frac 1n, \\
\textup{(ii)} & \quad R^k \frac TN \gg n \iff T \gg R^{-k}Nn, \\
\textup{(iii)} & \quad R^k \gg \frac TN R^{2k-1}
\iff
T \ll R^{-(k-1)}N , \\
\textup{(iv)} & \quad R^k N^{-1} \gg T^3
\iff
T \ll R^{\frac{k}{3}} N^{-\frac{1}{3}}
,\\
\textup{(v)} & \quad T \ll \min\big( R^{- \frac{k(k-1)}{k+1}} N^{\frac{k-1}{k+1}}, R^k N^{-1}\big)
, \\
\textup{(vi)} & \quad N^{-1} \ll T \ll 1
\end{align*}
and $s<0$. Note that (v) comes from \eqref{cond2} and \eqref{cond5}, while (vi) comes from the conditions in Proposition~\ref{PROP:NewEst1} and Lemmas~\ref{lem:it2t} and \ref{lem:ujit1}.

The conditions
(ii)--(vi) are equivalent to the following:
\begin{equation}
\max \Big( R^{-k}N, N^{-1} \Big)
\ll T
\ll \min \Big( R^{-(k-1)}N, R^{\frac{k}{3}} N^{-\frac{1}{3}}, R^k N^{-1}, R^{- \frac{k(k-1)}{k+1}} N^{\frac{k-1}{k+1}},  1 \Big).
\notag
\end{equation}
Then, $R$ needs to satisfy
\[
N^{\frac 1k} \ll R \ll N^{\frac{2}{k-1}}.
\]
With condition (i),
we can choose appropriate $T$ and $R$
when
$s<-\frac 1k$.
Namely, we obtain the norm inflation at general initial data with infinite loss of regularity for any $d \geq 1$, $2 \leq k \leq 5$ and $s<- \frac 1k$.

\begin{remark}
\label{REM:ni34}
\rm
Since
a condition of the convergence of the power series is strong for $A \gg 1$,
the argument in this subsection does not yield
a better ill-posedness result in Theorem \ref{THM:main} for $d=1$ and $k=3,4$.
Indeed,
we can not choose appropriate $R$, $T$ and $A$ for $d=1$ and $s>-\frac 1k$ satisfying
\begin{align*}
\textup{(i)} & \quad RN^s A^{\frac 12} \ll 1
\iff
R \ll N^{-s} A^{-\frac 12},
\\
\textup{(ii)} & \quad R^k \frac{A^{k-1}}N T^{-s+\frac 12} \gg 1
\iff
T \gg (R^{-k} N A^{-k+1})^{\frac 2{1-2s}}
, \\
\textup{(iii)} & \quad
T^2 R^{k-1} A^{k-1} (TN)^{-\frac{k-1}k}
\ll 1
\iff
T \ll (R^{-k} N A^{-k})^{\frac{k-1}{k+1}}
,
\\
\textup{(iv)} & \quad
A^{-1} \les T \ll 1.
\end{align*}
Here, (i) and (ii) come from \eqref{it04a0NormHs} and \eqref{EQ:NewEst1}, respectively.
Moreover,
by Lemma \ref{lem:ujit1},
(iii) ensures the convergence of the series.
As in the proof of \eqref{THM:mainC} in Theorem \ref{THM:main},
we need to take (iv).

By (ii) and (iii),
$R, N, A$ need to satisfy
\[
(R^{-k} N A^{-k+1})^{\frac 2{1-2s}} \ll (R^{-k} N A^{-k})^{\frac{k-1}{k+1}}
\iff
R \gg N^{\frac 1k} A^{\frac{(k-1)(-2ks-k-2)}{k(2(k-1)s+k+3)}},
\]
where we note that $2(k-1)s+k+3>0$ when $s>-\frac {k+3}{2(k-1)} = -\frac 12 - \frac 2{k-1}$.
From (i), $N$ and $A$ should be
\[
N^{\frac 1k} A^{\frac{(k-1)(-2ks-k-2)}{k(2(k-1)s+k+3)}} \ll N^{-s} A^{-\frac 12}
\iff
N^{s+\frac 1k} \ll A^{\frac{2k(k-1)s+k^2-k-4}{2k(2(k-1)s+k+3)}}.
\]
This condition becomes
\[
N^{s+\frac 13} \ll A^{\frac{6s+1}{6(2s+3)}}
\quad \text{and} \quad
N^{s+\frac 14} \ll A^{\frac{3s+1}{6s+7}}
\]
for $k=3$ and $k=4$, respectively.
When $k=3$,
we can not choose $A \in (1,N)$ for $s \ge -\frac 13$.
On the other hand,
when $k=4$ and $s \ge - \frac 14$,
we can choose appropriate $A \in (1,N)$ only for $-\frac 14  \le s < - \frac 16 = s_{\rm scal}$,
where $s_{\rm scal}$ is defined in \eqref{EQ:sc}.
However,
the norm inflation in this range is already proved in Section \ref{Sec3}.
\end{remark}

\section{Well-posedness of vNLW}
\label{sec:WPofvNLW}

In this section, we study the well-posedness of \eqref{vNLW} in negative Sobolev spaces.
To do so, we use Schauder estimate.
See Lemma 10 in \cite{LO22}.

\begin{proposition}
\label{prop:Schauder}
For any $\s \ge 0$, $1 \le p \le q \le \infty$ and $0<t<1$,
we have
\[
\big\| D^\s P(t)u_0 \big\|_{L^q (\M^d)}
\les
t^{-\s -d ( \frac 1p - \frac 1q)}
\| u_0 \|_{L^p (\M^d)},
\]
where $P(t)$ is defined by \eqref{EQ:DefiPoissonKer}.
\end{proposition}

For $0<T<1$ and $-\frac d2<s \le 0$, we denote
\begin{equation}
\notag
\| u \|_{Y(T)} := \sup_{0<t<T} t^{-s+d(\frac 12 - \frac 1k)} \| u(t) \|_{\dot H^{d(\frac 12 - \frac 1k)}}
\end{equation}

\noi and
\begin{equation}
\notag
\| u \|_{X(T)} := \| u \|_{C_T H^s} + \| u \|_{Y(T)},
\end{equation}

\noi where we denoted $C_T H^s \coloneqq C([0,T], H^s (\M^d))$.
We define
\[
X(T) := \big\{ u \in C([0,T]; H^s(\M^d)) \cap C((0,T); \dot H^{d (\frac 12 - \frac 1k)}(\M^d)) \, | \
\| u \|_{X(T)} < \infty \big\}.
\]

We use a contraction argument on the ball $B^X_{2R} \coloneqq \{ u \in X(T) \, | \ \|u\|_{X(T)} \le 2R \}$. To do so, let us define the following functional on $X(T)$:
\[
\G(u)(t)
:= V(t) \vec{u}_0 + I_k (u)(t),
\]
where $V(t)$ and $I_k$ are defined in \eqref{EQ:DefiV} and \eqref{EQ:DefiDuhamelOp}, respectively.

From Proposition~\ref{prop:Schauder}, we have
$$
\| V(t) \vec u_0 \|_{X(T)}
\les
\| \vec u_0 \|_{\H^s} \leq R.
$$

\noi Let us now estimate $\| I_k (u)(t)\|_{X(T)}$ by studying separately the two norms.
Note that, since $k$ and $d$ satisfy \eqref{WPA}, we have
\begin{equation}
\label{WPcondreg}
\max \Big( s_{\rm scal}, -\frac d2 \Big) \leq s_{\rm vis} < s \leq 0.
\end{equation}

\noi On one hand, since $- \frac d2 < s \leq 0$, we have by Sobolev embedding $W^{-s, \frac{2d}{d-2s}}(\M^d) \embeds L^2(\M^d)$ with $\frac{2d}{d-2s} >1$. Therefore, we have by Proposition \ref{prop:Schauder} for any $0 < t < T$
\begin{align*}
\| I_k (u) (t) \|_{H^s}
&\les
\int_0^t (t-t') \big\| P(t-t') (u^k)(t') \big\|_{H^s} \mathrm{d}t' \\
&\les
\int_0^t (t-t') \big\| P(t-t') (u^k)(t') \big\|_{L^{\frac{2d}{d-2s}}} \mathrm{d}t' \\
&\les
\int_0^t (t-t')^{1-d ( 1- \frac{d-2s}{2d})} \big\| u^k(t') \big\|_{L^1} \mathrm{d}t'
\\
&\les
\int_0^t (t-t')^{-s-\frac d2+1} \| u(t') \|_{L^k}^k \mathrm{d}t' \\
&\les
\int_0^t (t-t')^{-s-\frac d2+1} t'^{k \{ s-d(\frac 12-\frac 1k)\}} \mathrm{d}t'
\| u \|_{Y(T)}^k.
\end{align*}

\noi Note that, by combining the assumption \eqref{WPA} and \eqref{WPcondreg}, this last time integral converges and we get
$$
\| I_k (u) \|_{C_T H^s} \les T^{\ta_1} \|u\|^k_{Y(T)}
$$

\noi for some $\ta_1 >0$.
Similarly,
we have
\begin{align*}
\| I_k (u) (t) \|_{\dot H^{d(\frac 12 - \frac 1k)}}
&\les
\int_0^t (t-t') \big\| P(t-t') (u^k)(t') \big\|_{\dot H^{d(\frac 12 - \frac 1k)}} \mathrm{d}t'
\\
&\les
\int_0^t (t-t')^{1-d(1-\frac 1k)} \big\| u^k(t') \big\|_{L^1} \mathrm{d}t'
\\
&\les
\int_0^t (t-t')^{1-d(1-\frac 1k)} t'^{k \{ s-d(\frac 12-\frac 1k)\}} \mathrm{d}t'
\| u \|_{Y(T)}^k.
\end{align*}

\noi According again to the assumption \eqref{WPA} and \eqref{WPcondreg}, our last time integral converges and yields
$$
\| I_k (u)\|_{Y(T)} \les T^{\ta_2} \|u\|^k_{Y(T)}
$$

\noi for some $\ta_2 > 0$. Choosing $\ta = \min( \ta_1, \ta_2)$ and $0<T<1$ such that $T^\ta \ll \frac{1}{2^{k+1}R^{k-1}}$, we get
$$
\|\G[u]\|_{X(T)} \leq 2R
$$

\noi and $\G$ maps the ball $B^X_{2R}$ into itself. Using the linearity of $W(t)$ and H\"older's inequality, we prove in a similar way that $\G$ is a contraction. By Banach fixed point theorem, there exists a unique fixed point to $\G$ in $B^X_{2R}$.
Since the continuity with respect to the initial data comes from an standard argument,
we omit the details here.

\begin{remark}
\rm
\label{REM:wp1d}
When $d=1$,
we also obtain the well-posedness in $H^s(\M)$ for $s> \frac 12- \frac 2k = s_{\rm vis}$.
Indeed,
when $\frac 12- \frac 2k < s \le \frac 12 - \frac 1k$,
it follows from \eqref{EQ:DefiDuhamelOp}, $L^1(\M) \hookrightarrow H^{-\frac 12-\eps} (\M)$
and $H^{\frac 12-\frac 1k}(\M) \hookrightarrow L^k(\M)$
 that
\begin{align*}
\| I_k (u) (t) \|_{H^{\frac 12-\frac 1k}}
&\les
\int_0^t \big\| (u^k)(t') \big\|_{H^{-\frac 12 - \frac 1k}} dt'
\les
\int_0^t \big\| (u^k)(t') \big\|_{L^1} dt'
\les
\int_0^t \| u(t') \|_{L^k}^k dt'
\\
&\les
\int_0^t \| u(t') \|_{\dot H^{\frac 12 - \frac 1k}}^k dt'
\les
\int_0^t t'^{(s-\frac 12+\frac 1k)k} dt' \| u \|_{Y(T)}^k
\les
T^{\theta_1} \| u \|_{Y(T)}^k
\end{align*}
with
$\theta_1 := (s-\frac 12+\frac 1k)k+1$.
Note that
$s>\frac 12 -\frac 2k$ implies
$\theta_1>0$.
Moreover,
when $\frac 12 - \frac 1k < s< \frac 12$,
the same calculation above yields that
\begin{align*}
\| I_k (u) (t) \|_{H^s}
\les
\int_0^t \big\| (u^k)(t') \big\|_{L^1} dt'
\les
\int_0^t \| u(t') \|_{L^k}^k dt'
\les T \| u \|_{C_T H^s}^k.
\end{align*}
Since
the case $s \ge \frac 12$ is easily treated,
we omit the details here.
\end{remark}

\begin{ackno}\rm
The authors would like to thank Tadahiro Oh for suggesting this problem. P.dR. acknowledges support from Tadahiro Oh's ERC grant (no. 864138 ``SingStochDispDyn"). M.O.~ was supported by JSPS KAKENHI Grant number JP23K03182.

\end{ackno}

\end{document}